\documentclass[12pt]{article}
\usepackage{amsmath, amsthm, amssymb}
\usepackage[all]{xy}
\usepackage{fullpage}
\usepackage{titlesec}
\usepackage{mathrsfs}
\usepackage{amsbsy}
\usepackage{turnstile}
\usepackage{bbm}
\usepackage[pdftex,bookmarks=true]{hyperref}

\titleformat{\section}{\normalsize\bfseries}{\thesection}{1em}{}
\titleformat{\subsection}{\normalsize\bfseries}{\thesubsection}{1em}{}

\numberwithin{equation}{subsection}

\theoremstyle{plain}

\newtheorem{PropSub}[subsection]{Proposition}
\newtheorem{LemSub}[subsection]{Lemma}
\newtheorem{CorSub}[subsection]{Corollary}
\newtheorem{ThmSub}[subsection]{Theorem}
\newtheorem{UnlabelledPropSub}[subsection]{}

\theoremstyle{definition}

\newtheorem{DefSub}[subsection]{Definition}
\newtheorem{ExaSub}[subsection]{Example}
\newtheorem{RemSub}[subsection]{Remark}
\newtheorem{ParSub}[subsection]{}

\newcommand{\bref}[1]{\textnormal{\textbf{\ref{#1}}}}
\newcommand{\pbref}[1]{\textnormal{(\textbf{\ref{#1}})}}

\newcommand{\bD}{\ensuremath{\boldsymbol{D}}}

\newcommand{\bF}{\ensuremath{\boldsymbol{F}}}
\newcommand{\bG}{\ensuremath{\boldsymbol{G}}}
\newcommand{\bH}{\ensuremath{\boldsymbol{H}}}

\newcommand{\bJ}{\ensuremath{\boldsymbol{J}}}
\newcommand{\bK}{\ensuremath{\boldsymbol{K}}}

\newcommand{\bP}{\ensuremath{\boldsymbol{P}}}
\newcommand{\bQ}{\ensuremath{\boldsymbol{Q}}}
\newcommand{\bS}{\ensuremath{\boldsymbol{S}}}
\newcommand{\bT}{\ensuremath{\boldsymbol{T}}}

\newcommand{\A}{\ensuremath{\mathscr{A}}}
\newcommand{\B}{\ensuremath{\mathscr{B}}}
\newcommand{\C}{\ensuremath{\mathscr{C}}}
\newcommand{\D}{\ensuremath{\mathscr{D}}}
\newcommand{\E}{\ensuremath{\mathscr{E}}}

\newcommand{\J}{\ensuremath{\mathscr{J}}}
\newcommand{\K}{\ensuremath{\mathscr{K}}}
\newcommand{\sL}{\ensuremath{\mathscr{L}}}
\newcommand{\M}{\ensuremath{\mathscr{M}}}
\newcommand{\sS}{\ensuremath{\mathscr{S}}}
\newcommand{\U}{\ensuremath{\mathscr{U}}}
\newcommand{\V}{\ensuremath{\mathscr{V}}}
\newcommand{\W}{\ensuremath{\mathscr{W}}}
\newcommand{\X}{\ensuremath{\mathscr{X}}}

\newcommand{\tL}{\ensuremath{\widetilde{\mathscr{L}}}}

\newcommand{\eA}{\ensuremath{\pmb{\mathscr{A}}}}
\newcommand{\eB}{\ensuremath{\pmb{\mathscr{B}}}}
\newcommand{\eC}{\ensuremath{\pmb{\mathscr{C}}}}

\newcommand{\eL}{\ensuremath{\pmb{\mathscr{L}}}}

\newcommand{\teL}{\ensuremath{\pmb{\widetilde{\mathscr{L}}}}}

\newcommand{\aeB}{\ensuremath{\underline{\pmb{\mathscr{B}}}}}
\newcommand{\aeC}{\ensuremath{\underline{\pmb{\mathscr{C}}}}}

\newcommand{\aeX}{\ensuremath{\underline{\pmb{\mathscr{X}}}}}
\newcommand{\aeL}{\ensuremath{\underline{\pmb{\mathscr{L}}}}}
\newcommand{\aeU}{\ensuremath{\underline{\pmb{\mathscr{U}}}}}
\newcommand{\aeV}{\ensuremath{\underline{\pmb{\mathscr{V}}}}}
\newcommand{\aeW}{\ensuremath{\underline{\pmb{\mathscr{W}}}}}

\newcommand{\aetL}{\ensuremath{\widetilde{\underline{\pmb{\mathscr{L}}}}}}

\newcommand{\VCAT}{\ensuremath{\V\textnormal{-\textbf{CAT}}}}
\newcommand{\WCAT}{\ensuremath{\W\textnormal{-\textbf{CAT}}}}
\newcommand{\XCAT}{\ensuremath{\X\textnormal{-\textbf{CAT}}}}
\newcommand{\LCAT}{\ensuremath{\sL\textnormal{-\textbf{CAT}}}}

\newcommand{\MMCCAATT}{\ensuremath{\mathfrak{MCAT}}}
\newcommand{\SSMMCCAATT}{\ensuremath{\mathfrak{SMCAT}}}
\newcommand{\CCllSSMMCCAATT}{\ensuremath{\mathfrak{ClSMCAT}}}
\newcommand{\CCllCCAATT}{\ensuremath{\mathfrak{ClCAT}}}
\newcommand{\TWOCCAATT}{\ensuremath{\mathfrak{2CAT}}}

\newcommand{\CC}{\ensuremath{\mathbb{C}}}
\newcommand{\DD}{\ensuremath{\mathbb{D}}}
\newcommand{\tDD}{\ensuremath{\widetilde{\mathbb{D}}}}
\newcommand{\HH}{\ensuremath{\mathbb{H}}}
\newcommand{\tHH}{\ensuremath{\widetilde{\mathbb{H}}}}
\newcommand{\RR}{\ensuremath{\mathbb{R}}}
\newcommand{\TT}{\ensuremath{\mathbb{T}}}
\newcommand{\ONEONE}{\ensuremath{\mathbbm{1}}}

\newcommand{\tbH}{\ensuremath{\boldsymbol{\widetilde{H}}}}

\newcommand{\tD}{\ensuremath{\widetilde{D}}}
\newcommand{\tH}{\ensuremath{\widetilde{H}}}

\newcommand{\mor}{\ensuremath{\operatorname{\textnormal{\textsf{mor}}}}}
\newcommand{\Epi}{\ensuremath{\operatorname{\textnormal{\textsf{Epi}}}}}
\newcommand{\Mono}{\ensuremath{\operatorname{\textnormal{\textsf{Mono}}}}}
\newcommand{\StrMono}{\ensuremath{\operatorname{\textnormal{\textsf{StrMono}}}}}

\newcommand{\Ev}{\ensuremath{\textnormal{\textsf{Ev}}}}

\newcommand{\Set}{\ensuremath{\operatorname{\textnormal{\textbf{Set}}}}}

\newcommand{\RVect}{\ensuremath{R\textnormal{-\textbf{Vect}}}}
\newcommand{\Top}{\ensuremath{\operatorname{\textnormal{\textbf{Top}}}}}
\newcommand{\Conv}{\ensuremath{\operatorname{\textnormal{\textbf{Conv}}}}}
\newcommand{\Adj}{\ensuremath{\operatorname{\textnormal{\textbf{Adj}}}}}
\newcommand{\Mnd}{\ensuremath{\operatorname{\textnormal{\textbf{Mnd}}}}}

\newcommand{\subs}{\ensuremath{\subseteq}}

\newcommand{\op}{\ensuremath{\textnormal{op}}}

\newcommand{\btimes}{\ensuremath{\boxtimes}}

\begin{document}

\author{\normalsize  Rory B.B. Lucyshyn-Wright\thanks{Partial financial assistance by the Ontario Graduate Scholarship program is gratefully acknowledged.}\let\thefootnote\relax\footnote{Keywords: commutative monads; Fubini theorem; functionals; measures; distributions; closed categories; symmetric monoidal categories; enriched categories; symmetric monoidal adjunctions; convergence vector spaces}
\\
\small York University, 4700 Keele St., Toronto, ON, Canada M3J 1P3}

\title{\large \textbf{A general Fubini theorem for the Riesz paradigm}}

\date{}

\maketitle

\abstract{We prove an abstract Fubini-type theorem in the context of monoidal and enriched category theory, and as a corollary we establish a Fubini theorem for integrals on arbitrary convergence spaces that generalizes (and entails) the classical Fubini theorem for Radon measures on compact Hausdorff spaces.  Given a symmetric monoidal closed adjunction satisfying certain hypotheses, we show that an associated \textit{monad of natural distributions} $\DD$ is commutative.  Applying this result to the monoidal adjunction between convergence spaces and convergence vector spaces, the commutativity of $\DD$ amounts to a Fubini theorem for continuous linear functionals on the space of scalar functions on an arbitrary convergence space.
}

\section{Introduction} \label{sec:intro}

The aim of this paper is to establish a vast generalization of the classical Fubini theorem of Bourbaki for Radon measures on compact Hausdorff spaces (\cite{Bou}, Ch. III, \S 5, proved also by Edwards \cite{Edw}).  We prove a Fubini-type theorem that applies to a much wider class of spaces, including not only all topological spaces but also arbitrary \textit{convergence spaces} (see \cite{BeBu}).  Further, we reason in an abstract context of monoidal and enriched category theory, thus proving an abstract result on symmetric monoidal adjunctions and commutative monads that is applicable, in particular, to the axiomatic or \textit{synthetic} study of functional analysis in a closed category initiated by Lawvere \cite{Law:IntroCatsContPhys,Law:VoFu} and Kock \cite{Kock:ProResSynthFuncAn}.  Bourbaki's Fubini theorem for compact spaces is obtained as a corollary and so is proved by entirely new means.

By the Riesz representation theorem (which exemplifies the \textit{Riesz paradigm} of Lawvere \cite{Law:AxEd}), there is a bijective correspondence between $R$-valued Radon measures (for $R = \RR$ or $\CC$) on a compact Hausdorff space $X$ and continuous linear functionals $[X,R] \rightarrow R$ on the Banach space $[X,R]$ of continuous $R$-valued functions.  Working with such functionals $\mu:[X,R] \rightarrow R$ rather than their associated measures, we employ the notation $\int_x f(x) \;d\mu$ or $\int f \;d\mu$ for the value $\mu(f)$ of $\mu$ at $f \in [X,R]$.  Thus the classical Fubini theorem of Bourbaki \cite{Bou}, when restricted from locally compact to compact Hausdorff spaces $X$, $Y$, may be stated as follows:
\begin{UnlabelledPropSub} \label{prop:fubini}
Given continuous linear functionals $\mu:[X,R] \rightarrow R$ and $\nu:[Y,R] \rightarrow R$, there is a unique continuous linear functional $\mu \otimes \nu:[X \times Y, R] \rightarrow R$ such that
\begin{equation}\label{eqn:fubini}\int_y \int_x f(x,y) \;d\mu d\nu = \int f \;d(\mu \otimes \nu) = \int_x \int_y f(x,y) \;d\nu d\mu\end{equation}
for all continuous functions $f:X \times Y \rightarrow R$.  In particular, each integrand is a continuous function.
\end{UnlabelledPropSub}
This theorem has a natural interpretation in the language of \textit{cartesian closed categories}.  In particular, let us embed the category of topological spaces $\Top$ into the cartesian closed category $\X = \Conv$ of convergence spaces and continuous maps, so that we have for all objects $X$, $Z$ of $\X$ an \textit{exponential} or \textit{function space} $[X,Z]$.  In the case that $X$ is a compact Hausdorff space and $Z = R$, the function space $[X,R]$ coincides with the classical space $[X,R]$ considered above, and we can rewrite the Fubini equation \eqref{eqn:fubini} in the notation of \textit{lambda calculus} as
\begin{equation}\label{eqn:lambda_fubini}\nu(\lambda y.\mu(\lambda x.f(x,y))) = (\mu \otimes \nu)(f) = \mu(\lambda x.\nu(\lambda y.f(x,y)))\;.\end{equation}
From this perspective, the continuity of the integrands in \bref{prop:fubini} is automatic, as is the uniqueness of $\mu \otimes \nu$, and we have two natural candidates for $\mu \otimes \nu$, given by the leftmost and rightmost expressions in \eqref{eqn:lambda_fubini}, so that the Fubini theorem may be distilled to the statement that these are equal.

We prove the following:
\begin{ThmSub}\label{thm:fubini_for_conv}
The Fubini theorem of \bref{prop:fubini} holds for \textit{arbitrary} convergence spaces $X$ and $Y$, when $[X,R]$ and $[Y,R]$ are interpreted as the associated function spaces in the category of convergence spaces.
\end{ThmSub}

This theorem is obtained as a corollary to an abstract result concerning an arbitrary symmetric monoidal adjunction 
\begin{equation}\label{eqn:smc_adj}
\xymatrix {
\X \ar@/_0.5pc/[rr]_F^(0.4){}^(0.6){}^{\top} & & \sL \ar@/_0.5pc/[ll]_G
}
\end{equation}
where $\X$ and $\sL$ are symmetric monoidal closed categories.  We let $R$ denote the unit object of $\sL$.  In our key example, where $\X = \Conv$, we take $\sL = \RVect(\X)$ to be the category of $R$-vector-space objects in $\X$ (or \textit{convergence vector spaces}) and let $G$ be the forgetful functor.  Quite generally, the symmetric monoidal adjunction \eqref{eqn:smc_adj} automatically acquires the structure of an $\X$-enriched adjunction, with $\sL$ a \textit{cotensored} $\X$-category.  In our example, the function spaces $[X,R]$ serve as cotensors of $R$ in this $\X$-enriched category.  In the general setting, the cotensors $[X,R]$ in $\sL$ give rise to an $\X$-enriched monad $\DD$ on $\X$ whose underlying $\X$-enriched functor $D:\X \rightarrow \X$ is given by $DX = \sL([X,R],R)$.  In our example, $DX$ is the canonical space of continuous linear functionals $\mu:[X,R] \rightarrow R$.  Referring to the elements of $DX$ as \textit{natural distributions}, we call $\DD$ the \textit{natural distribution monad}.

Kock has considered the monad $\DD$, in a slightly different setting, under the name of \textit{the Schwartz double-dualization monad} \cite{Kock:ProResSynthFuncAn,Kock:Dist}.  Kock \cite{Kock:Dist} has claimed that for an \textit{arbitrary} $\X$-monad $\TT$ on a cartesian closed category $\X$, the statement that $\TT$ is \textit{commutative} \cite{Kock:Comm} is a form of Fubini's Theorem.  In the earlier paper \cite{Kock:ProResSynthFuncAn}, Kock had made this connection more explicit in the case of the Schwartz double-dualization monad on a ringed topos, with reference to the terms appearing in the Fubini equations \eqref{eqn:fubini}, \eqref{eqn:lambda_fubini}.  We thus reduce our task of proving our Fubini theorem for convergence spaces \pbref{thm:fubini_for_conv} to the problem of proving that the monad $\DD$ on $\Conv$ is commutative.

However, inherent in Kock's investigations was the general conclusion (for general $\X$) that $\DD$ ``is not commutative'' (\cite{Kock:Dist}, pg. 97), whereas we show that the monad $\DD$ on $\X = \Conv$ \textit{is} commutative, by means of the following result in our general setting:

\begin{ThmSub}\label{thm:general}
Suppose given a symmetric monoidal adjunction as in \eqref{eqn:smc_adj}, with $\sL$ locally small and finitely well-complete.  Suppose also that each cotensor $[X,R]$ in the $\X$-category $\sL$ is reflexive.  Then $\DD$ is commutative.
\end{ThmSub}

Here, the notion of \textit{finite well-completeness} is a mild requirement of the existence of certain (inverse) limits in $\sL$.  The condition of \textit{reflexivity} of an object $E$ of $\sL$ is the natural one afforded by the symmetric monoidal closed structure on $\sL$, namely that the canonical morphism $E \rightarrow E^{**}$ into the \textit{double dual} $E^{**} = \underline{\mathscr{L}}(\underline{\mathscr{L}}(E,R),R)$ is an isomorphism, where $\underline{\mathscr{L}}(-,-)$ here denotes the `internal-hom' functor of $\sL$.   In our example, $E^* = \underline{\mathscr{L}}(E,R)$ is called the \textit{continuous dual} of the convergence vector space $E$, and indeed it was shown by Butzmann \cite{Bu} that the convergence vector spaces $[X,R]$ are reflexive in the given sense, so that the needed commutativity of $\DD$ in this context is obtained.

The proof of Theorem \bref{thm:general} is quite technical and yet relatively simple in its overall form, and so we now give an informal sketch.  The first key observation is that since $[X,R] \cong \underline{\mathscr{L}}(FX,R) = (FX)^*$, we have that $[X,R]^* \cong (FX)^{**}$.  In our example, this means that the convergence vector space of natural distributions $[X,R]^*$ is isomorphic to the double dual of the free convergence vector space $FX$ on $X$.  In the general setting, the double dualization endofunctor $H := (-)^{**}$ on $\sL$ underlies an $\sL$-enriched monad $\HH$ on $\sL$, and we find that, up to isomorphism, $\DD$ is induced by the composite $\X$-enriched adjunction
\begin{equation}\label{eq:fg_em_composite_adj}
\xymatrix {
\X \ar@/_0.5pc/[rr]_F^(0.4){}^(0.6){}^{\top} & & \sL \ar@/_0.5pc/[ll]_G \ar@/_0.5pc/[rr]_{F^{\HH}}^(0.4){}^(0.6){}^{\top} & & {\sL^{\HH}\;,} \ar@/_0.5pc/[ll]_{G^{\HH}}
}
\end{equation}
in which the rightmost adjunction is the Eilenberg-Moore adjunction for the double-dualization monad $\HH$.

Kock \cite{Kock:SmComm} showed that, up to a bijection, commutative $\X$-enriched monads are the same as symmetric monoidal monads.  But the leftmost adjunction $F \dashv G$ in \eqref{eq:fg_em_composite_adj} is symmetric monoidal, so our strategy is to \textit{replace} the rightmost adjunction by one that is symmetric monoidal --- without affecting the $\X$-enriched monad $\DD$ induced by the composite.

To this end, we apply the work of Day \cite{Day:AdjFactn} in order to factorize the $\sL$-enriched Eilenberg-Moore adjunction $F^{\HH} \dashv G^{\HH}$ as a composite consisting of an $\sL$-enriched reflection 
\begin{equation}\label{eqn:completion_refl}
\xymatrix {
\sL \ar@/_0.5pc/[rr]^(0.4){}^(0.6){}^{\top} & & {\tL} \ar@{_{(}->}@/_0.5pc/[ll]
}
\end{equation}
followed by a conservative left adjoint.  We call the objects of the reflective subcategory $\tL$ the \textit{(functionally) complete} objects of $\sL$, and we call the induced idempotent monad $\tHH$ on $\sL$ the \textit{(functional) completion monad}.  Since $\tL$ is an $\sL$-enriched reflective subcategory of $\sL$, $\tL$ is closed under cotensors in $\sL$, and so it follows from Day's work on closed reflections \cite{Day:Refl} that $\tL$ is symmetric monoidal closed and that the adjunction \eqref{eqn:completion_refl} is symmetric monoidal.  Replacing the rightmost adjunction in the composite \eqref{eq:fg_em_composite_adj} by this reflection, we thus obtain a composite $\X$-enriched adjunction
\begin{equation}\label{eq:fg_compl_composite_adj}
\xymatrix {
\X \ar@/_0.5pc/[rr]_F^(0.4){}^(0.6){}^{\top} & & \sL \ar@/_0.5pc/[ll]_G \ar@/_0.5pc/[rr]^(0.4){}^(0.6){}^{\top} & & {\tL} \ar@{_{(}->}@/_0.5pc/[ll]
}
\end{equation}
whose factors' underlying ordinary adjunctions are symmetric monoidal.  

Now the key step is to show that the $\X$-monad induced by this composite \eqref{eq:fg_compl_composite_adj} is isomorphic to $\DD$.  At the level of objects, this amounts to the statement that the space of natural distributions $[X,R]^* \cong (FX)^*$ is isomorphic to the \textit{completion} $\tH FX$ of the `free span' $FX$ of $X$.  It is at this stage that we make use of the hypothesis that the cotensors $[X,R]$ are reflexive.

Finally, we would like to reason that $\DD$ is induced (up to isomorphism) by the composite symmetric monoidal adjunction \eqref{eq:fg_compl_composite_adj} and so is a symmetric monoidal monad and therefore a commutative monad.  In the interest of rigour, however, we must (in effect) verify that the resulting commutative $\X$-enriched \textit{structure} on (the ordinary monad underlying) $\DD$ coincides with the given $\X$-enriched structure carried by $\DD$.

This paper is organized into two parts:  Whereas \small{P\textsc{art} II provides complete proofs of Theorems \bref{thm:fubini_for_conv} and \bref{thm:general} along the above lines, \small{P\textsc{art} I} develops several topics in enriched and monoidal category theory that are needed for \small{P\textsc{art} II} but are of general applicability.  In particular, in \bref{sec:enr_mon_func} and \bref{sec:enr_smc_adj} we study the enriched structure canonically associated to an arbitrary symmetric monoidal adjunction of closed categories \eqref{eqn:smc_adj}:  its relation to change-of-base for enriched categories, its properties with regard to composition of adjunctions, and, in \bref{sec:comm_mnd_smc_adj}, its relation to Kock's bijection \cite{Kock:SmComm} between symmetric monoidal monads and commutative enriched monads.  In \bref{sec:smc_refl} we recall a result of Day \cite{Day:Refl} on symmetric monoidal closed reflections and study its relation to the canonical enriched structure of \bref{sec:enr_smc_adj}.  In \bref{sec:enr_orth_fwc} we study enriched notions of \textit{orthogonality}, of \textit{factorization system}, and of \textit{finite well-completeness} in the sense of \cite{CHK}, showing that enriched finite well-completeness reduces to ordinary in the case of the base category $\V$.  This is followed in \bref{sec:enr_adj_factn} by a treatment of enriched adjoint functor factorization that builds upon the work of Day \cite{Day:AdjFactn}.  Whereas Day's explicit aim in \cite{Day:AdjFactn} was to provide an approach for proving an enriched analogue of a factorization result of Applegate and Tierney, such a result is neither precisely stated nor proved there. Rather, Day proves key lemmas that allow a proof of such a result.  We fill this gap by giving a statement and proof of the resulting adjoint factorization theorem \pbref{thm:adj_factn}.

\newpage
\vspace{1pc}
\begin{center}\normalsize{{P\textsc{art} I:  S\textsc{upporting} T\textsc{heory}}}\end{center}

\section{Notation and 2-categorical preliminaries}

Ordinary categories and functors, as well as monoidal such, are denoted by script letters and uppercase letters, respectively (e.g. $\A$, $F$), whereas $\V$-categories and $\V$-functors (for a monoidal category $\V$) are denoted by bold letters (e.g. $\eA$, $\bF$, respectively), with the underlying ordinary category or functor denoted by the corresponding non-bold letter.

Given a symmetric monoidal closed category $\V$, we denote by $\aeV$ the canonically associated $\V$-category whose underlying ordinary category is $\V$; in particular, the internal homs in $\V$ will therefore be denoted by $\aeV(V_1,V_2)$, whereas we reserve the square-bracket notation $[-,-]$ of \bref{sec:intro} for cotensors in a general $\V$-category.  The canonical `evaluation' morphisms $V_1 \otimes \aeV(V_1,V_2) \rightarrow V_2$ are denoted by $\Ev_{V_1 V_2}$, or simply $\Ev$.  We sometimes omit subscripts and names of morphisms when they are clear from the context.

We shall require the following basic results in the 2-categorical context:

\begin{PropSub} \label{thm:mon_func_detd_by_adj}
Let $f \nsststile{\epsilon}{\eta} g : B \rightarrow A$ be an adjunction in a 2-category $\K$.  Then there is an associated monoidal functor $[f,g] := \K(f,g) : \K(B,B) \rightarrow \K(A,A)$ with the following property:  For any adjunction $f' \nsststile{\epsilon'}{\eta'} g':C \rightarrow B$ with induced monad $\TT$ on $B$, the monad $[f,g](\TT)$ on $A$ is equal to the monad induced by the composite adjunction 
$$\xymatrix{A \ar@/_0.5pc/[rr]_f^(0.4){\eta}^(0.6){\epsilon}^{\top} & & B \ar@/_0.5pc/[ll]_g \ar@/_0.5pc/[rr]_{f'}^(0.4){\eta'}^(0.6){\epsilon'}^{\top} & & {C\;.} \ar@/_0.5pc/[ll]_{g'}}$$
\end{PropSub}
\begin{proof}
The monoidal structure on the functor $[f,g]$ consists of the morphisms $gh\epsilon hf:ghfgkf \rightarrow ghkf$ in $\K(A,A)$ (for all objects $h$, $k$ in $\K(B,B)$) and the morphism $\eta:1_A \rightarrow gf$ in $\K(A,A)$.  The verification is straightforward.
\end{proof}

\begin{ParSub} \label{par:adj_and_mnd}
Given objects $A$, $B$ in a 2-category $\K$, there is a category $\Adj_\K(A,B)$ whose objects are adjunctions $f \nsststile{\epsilon}{\eta} g : B \rightarrow A$ in $\K$ and whose morphisms $(\phi,\psi):(f \nsststile{\epsilon}{\eta} g) \rightarrow (f' \nsststile{\epsilon'}{\eta'} g')$ consist of 2-cells $\phi:f \rightarrow f'$ and $\psi:g \rightarrow g'$ such that $(\psi \circ \phi) \cdot \eta = \eta'$ and $\epsilon' \cdot (\phi \circ \psi) = \epsilon$.

There is a category $\Mnd_\K(A)$ whose objects are monads on $A$ and whose morphisms $\theta:(t,\eta,\mu) \rightarrow (t',\eta',\mu')$ consist of a 2-cell $\theta:t \rightarrow t'$ such that $\theta \cdot \eta = \eta'$ and $\mu' \cdot (\theta \circ \theta) = \theta \cdot \mu$.  The identity monad $\ONEONE_A$ is an initial object in $\Mnd_\K(A)$, since for each monad $\TT = (t,\eta,\mu)$ on $A$, the 2-cell $\eta$ is the unique monad morphism $\eta:\ONEONE_A \rightarrow \TT$.

There is a functor $\Adj_{\K}(A,B) \rightarrow \Mnd_\K(A)$ sending an adjunction to its induced monad and a morphism $(\phi,\psi):(f \nsststile{\epsilon}{\eta} g) \rightarrow (f' \nsststile{\epsilon'}{\eta'} g')$ to the morphism $\psi \circ \phi:\TT \rightarrow \TT'$ between the induced monads.  Hence, in particular, isomorphic adjunctions induce isomorphic monads.
\end{ParSub}

\begin{PropSub} \label{thm:uniq_adj}
Let $f \nsststile{\epsilon}{\eta} g$ and $f' \nsststile{\epsilon'}{\eta'} g$ be adjunctions, having the same right adjoint $g:B \rightarrow A$, in a 2-category $\K$.  Then these adjunctions are isomorphic and hence induce isomorphic monads on $A$.
\end{PropSub}
\begin{proof}
The 2-cell $\phi := \epsilon f' \cdot f \eta':f \rightarrow f'$ has inverse $\epsilon' f \cdot f' \eta$ (\cite{Gray}), and one checks that $(\phi,1_g)$ serves as the needed isomorphism of adjunctions.
\end{proof}

\begin{PropSub} \label{thm:mnd_morph_from_adj_factn}
Let $\xymatrix{A \ar@/_0.5pc/[rr]_{f''}^(0.4){\eta''}^(0.6){\epsilon''}^{\top} & & C \ar@/_0.5pc/[ll]_{g''}}$ be an adjunction, with induced monad $\TT''$, in a 2-category $\K$.  Let 
$\xymatrix{A \ar@/_0.5pc/[rr]_f^(0.4){\eta}^(0.6){\epsilon}^{\top} & & B \ar@/_0.5pc/[ll]_g \ar@/_0.5pc/[rr]_{f'}^(0.4){\eta'}^(0.6){\epsilon'}^{\top} & & C \ar@/_0.5pc/[ll]_{g'}}$
be adjunctions with $gg' = g''$, and let $\TT$ and $\TT'$ be the respective induced monads.  Then there is an associated monad morphism $i:\TT \rightarrow \TT''$.
\end{PropSub}
\begin{proof}
Let $f_c \nsststile{\epsilon_c}{\eta_c} g''$ be the composite adjunction, and let $\TT_c$ be its induced monad.  By \bref{thm:mon_func_detd_by_adj}, we have that $\TT_c = [f,g](\TT')$, whereas $\TT = [f,g](\ONEONE_A)$.  By applying $[f,g]$ to the monad morphism $\eta':\ONEONE_A \rightarrow \TT'$, we obtain a monad morphism
$$g \eta' f = [f,g](\eta') : \TT = [f,g](\ONEONE_A) \rightarrow [f,g](\TT') = \TT_c\;.$$
Also, by \bref{thm:uniq_adj}, there is an isomorphism of monads $\xi:\TT_c \rightarrow \TT''$, and we obtain a composite morphism of monads
\begin{equation}\label{eqn:comp_morph_mnds}i := (\TT \xrightarrow{g \eta' f} \TT_c \xrightarrow{\xi} \TT'')\;.\end{equation}
\end{proof}

\section{Enriched functors arising from monoidal functors} \label{sec:enr_mon_func}

\begin{ParSub} \label{par:ch_base}
Cruttwell \cite{Cr} defines a 2-functor $(-)_*:\MMCCAATT \rightarrow \TWOCCAATT$, from the 2-category of monoidal categories to the 2-category of 2-categories, sending each monoidal functor $M:\V \rightarrow \W$ to the \textit{change-of-base} functor $M_*:\VCAT \rightarrow \WCAT$.  For a $\V$-category $\eA$, the $\W$-category $M_*\eA$ has objects those of $\eA$ and homs given by
$$(M_*\eA)(A_1,A_2) = M\eA(A_1,A_2)\;\;\;\;(A_1,A_2 \in \eA)\;.$$
Given a symmetric monoidal functor $M:\V \rightarrow \W$ between closed symmetric monoidal categories, we also obtain a canonical $\W$-functor $\grave{M}:M_*\aeV \rightarrow \aeW$, given on objects just as $M$, and with each
$$\grave{M}_{V_1 V_2} : (M_*\aeV)(V_1,V_2) = M\aeV(V_1,V_2) \rightarrow \aeW(MV_1,MV_2)\;\;\;\;(V_1,V_2 \in \V)$$
gotten as the transpose of the composite
$$MV_1 \otimes M\aeV(V_1,V_2) \rightarrow M(V_1 \otimes \aeV(V_1,V_2)) \xrightarrow{M(\Ev)} MV_2\;.$$
\end{ParSub}

\begin{ParSub} \label{par:cl_func}
Let $\SSMMCCAATT$ be the 2-category of symmetric monoidal categories, and let $\CCllSSMMCCAATT$ be the full sub-2-category of $\SSMMCCAATT$ of with objects all closed symmetric monoidal categories.  Letting $\CCllCCAATT$ be the 2-category of closed categories \cite{EiKe}, there is a 2-functor $c:\CCllSSMMCCAATT \rightarrow \CCllCCAATT$, sending a symmetric monoidal functor $M:\V \rightarrow \W$ (with $\V$ and $\W$ closed symmetric monoidal categories) to the closed functor $cM:\V \rightarrow \W$ with the same underlying ordinary functor $M$ and the same unit morphism $I_{\W} \rightarrow MI_{\V}$, but with each structure morphism
$$M\aeV(V_1,V_2) \rightarrow \aeW(MV_1,MV_2)\;\;\;\;(V_1,V_2 \in \V)$$
equal to the the morphism
$$\grave{M}_{V_1 V_2} : (M_*\aeV)(V_1,V_2) \rightarrow \aeW(MV_1,MV_2)$$
associated to the $\W$-functor $\grave{M}:M_*\aeV \rightarrow \aeW$ of \bref{par:ch_base}.
\end{ParSub}

\begin{PropSub} \label{prop:comp_assoc_enr_fun}
Let $M:\U \rightarrow \V$, $N:\V \rightarrow \W$ be symmetric monoidal functors between closed symmetric monoidal categories.  Then the $\W$-functor
$$\grave{(NM)}:(NM)_*\aeU \rightarrow \aeW$$
is equal to the composite
$$N_*M_*\aeU \xrightarrow{N_*\grave{M}} N_*\aeV \xrightarrow{\grave{N}} \aeW\;.$$
\end{PropSub}
\begin{proof}
Both $\W$-functors are given on objects just as $NM$.  Letting $C$ be the given composite $\W$-functor, the structure morphisms $C_{U_1 U_2}$ of $C$ (where $U_1,U_2$ are objects of $N_*M_*\aeU$, equivalently, of $\U$) are the composites
$$NM\aeU(U_1,U_2) \xrightarrow{N(\grave{M}_{U_1 U_2})} N\aeV(MU_1,MU_2) \xrightarrow{\grave{N}_{MU_1 MU_2}} \aeW(NMU_1,NMU_2)\;,$$
but by \bref{par:cl_func}, we have that $\grave{M}_{U_1 U_2}$ and $\grave{N}_{MU_1 MU_2}$ are equally the structure morphisms of the closed functors $cM$ and $cN$, respectively, and so $C_{U_1 U_2}$ is the structure morphism of the composite $(cN)(cM)$ of these closed functors.  Since $(cN)(cM) = c(NM)$, $C_{U_1 U_2}$ is therefore equally the structure morphism of the closed functor $c(NM)$ associated to $NM$, which by \bref{par:cl_func} is equal to
$$\grave{(NM)}_{U_1 U_2}:((NM)_*\aeU)(U_1,U_2) \rightarrow \aeW(NMU_1,NMU_2)\;,$$
the structure morphism of the $\aeW$ functor $\grave{(NM)}$.
\end{proof}

\section{Enrichment of a symmetric monoidal closed adjunction} \label{sec:enr_smc_adj}

\begin{ParSub} \label{par:given_smcadj}
Let
\begin{equation}\label{eqn:smcadj}
\xymatrix {
\X \ar@/_0.5pc/[rr]_F^(0.4){\eta}^(0.6){\epsilon}^{\top} & & \sL \ar@/_0.5pc/[ll]_G
}
\end{equation}
be a symmetric monoidal adjunction, where $\X$ and $\sL$ are \textit{closed} symmetric monoidal categories.  By applying $(-)_*$ \pbref{par:ch_base} to this monoidal adjunction, we obtain an adjunction
$$
\xymatrix {
\XCAT \ar@/_0.5pc/[rr]_{F_*}^(0.4){\eta_*}^(0.6){\epsilon_*}^{\top} & & \LCAT \ar@/_0.5pc/[ll]_{G_*}
}
$$
in $\TWOCCAATT$.  We have an $\X$-functor $\grave{G}:G_*\aeL \rightarrow \aeX$ and an $\sL$-functor $\grave{F}:F_*\aeX \rightarrow \aeL$ \pbref{par:ch_base}, and we obtain an $\X$-functor $\acute{F}:\aeX \rightarrow G_*\aeL$ as the the transpose of $\grave{F}$ under the preceding adjunction.
\end{ParSub}

The following well-known results have the status of `folklore'; see, e.g., \S 4, Prop. 1 of \cite{EgMoSi} and, for a sketch of a proof, Theorem 3.7.10 of \cite{Rie}.

\begin{PropSub}\label{thm:assoc_enr_adj}
Let $F \nsststile{\epsilon}{\eta} G : \sL \rightarrow \X$ be a symmetric monoidal adjunction between closed symmetric monoidal categories (as in \bref{par:given_smcadj}).
\begin{enumerate}
\item There is an $\X$-adjunction
\begin{equation}\label{eqn:assoc_enr_adj}
\xymatrix {
\aeX \ar@/_0.5pc/[rr]_{\acute{F}}^(0.4){\eta}^(0.6){\epsilon}^{\top} & & G_*\aeL \ar@/_0.5pc/[ll]_{\grave{G}}
}
\end{equation}
whose underlying ordinary adjunction may be identified with $F \nsststile{\epsilon}{\eta} G$.
\item The $\X$-category $G_*\aeL$ is cotensored.  For all $X \in \X$ and $E \in \sL$, a cotensor $[X,E]$ in $G_*\aeL$ may be gotten as the internal hom $\aeL(FX,E)$ in $\sL$.
\end{enumerate}
\end{PropSub}
\begin{ParSub} \label{par:cotensor_dual_of_free}
Further to \bref{thm:assoc_enr_adj}.2, one may obtain the needed `hom-cotensor' $\X$-adjunction
$$
\xymatrix {
\aeX \ar@/_0.5pc/[rr]_{[-,E]}^(0.4){\delta}^(0.6){\sigma}^{\top} & & {(G_*\aeL)^\op} \ar@/_0.5pc/[ll]_{(G_*\aeL)(-,E)}
}
$$
as exactly the following composite $\X$-adjunction
$$
\xymatrix {
\aeX \ar@/_0.5pc/[rr]_{\acute{F}}^(0.4){\eta}^(0.6){\epsilon}^{\top} & & G_*\aeL \ar@/_0.5pc/[ll]_{\grave{G}} \ar@/_0.5pc/[rr]_{G_*(\aeL(-,E))}^(0.4){}^(0.6){}^{\top} & & {G_*(\aeL^\op)\;,} \ar@/_0.5pc/[ll]_{G_*(\aeL(-,E))}
}
$$
in which the rightmost adjunction is gotten by applying $G_*:\LCAT \rightarrow \XCAT$ to the $\sL$-adjunction
$$
\xymatrix {
\aeL \ar@/_0.5pc/[rr]_{\aeL(-,E)}^(0.4){}^(0.6){}^{\top} & & {\aeL^\op\;.} \ar@/_0.5pc/[ll]_{\aeL(-,E)}
}
$$
\end{ParSub}
\begin{ExaSub}
Our principal example of a situation as in \bref{thm:assoc_enr_adj} is provided by \bref{thm:conv_smadj}, where $\X$ and $\sL$ are the categories of \textit{convergence spaces} and \textit{convergence vector spaces}, respectively.
\end{ExaSub}

\begin{PropSub} \label{thm:compn_assoc_enr_adj}
Let $\xymatrix {\W \ar@/_0.5pc/[rr]_F^(0.4){}^(0.6){}^{\top} & & \V \ar@/_0.5pc/[ll]_G}$ and $\xymatrix {\V \ar@/_0.5pc/[rr]_{L}^(0.4){}^(0.6){}^{\top} & & \U \ar@/_0.5pc/[ll]_{R}}$ be symmetric monoidal adjunctions between symmetric monoidal closed categories.  Then the $\W$-adjunction
$$
\xymatrix {
\aeW \ar@/_0.5pc/[rr]_{\acute{(LF)}}^(0.4){}^(0.6){}^{\top} & & (GR)_*\aeU \ar@/_0.5pc/[ll]_{\grave{(GR)}}
}
$$
associated to the composite symmetric monoidal adjunction $LF \dashv GR$ is isomorphic to the composite $\W$-adjunction
$$
\xymatrix {
\aeW \ar@/_0.5pc/[rr]_{\acute{F}}^(0.4){\eta}^(0.6){}^{\top} & & G_*\aeV \ar@/_0.5pc/[ll]_{\grave{G}} \ar@/_0.5pc/[rr]_{G_*\acute{L}}^(0.4){}^(0.6){}^{\top} & & {G_*R_*\aeU\;,} \ar@/_0.5pc/[ll]_{G_*\grave{R}}
}
$$
\end{PropSub}
\begin{proof}
By \bref{prop:comp_assoc_enr_fun}, we deduce that the right adjoints of these $\W$-adjunctions are equal, and the result follows by \bref{thm:uniq_adj}.
\end{proof}

\section{Commutative monads and symmetric monoidal closed adjunctions} \label{sec:comm_mnd_smc_adj}

Let $\X = (\X,\btimes,I)$ be a closed symmetric monoidal category.

\begin{DefSub}\label{def:comm_mnd}(Kock \cite{Kock:Comm})
Let $\TT = (\bT,\delta,\kappa)$ be an $\X$-monad on $\aeX$.
\begin{enumerate}
\item For objects $X,Y$ in $\X$, we define morphisms
$$t'_{XY}:T X \btimes Y \rightarrow T(X \btimes Y)\;,\;\;\;\;t''_{XY}:X \btimes T Y \rightarrow T(X \btimes Y)$$
as the transposes of the following composite morphisms
$$Y \rightarrow \aeX(X,X \btimes Y)
\xrightarrow{\bT} \aeX(T X,T(X \btimes Y))\;,\;\;\;\;X \rightarrow \aeX(Y,X \btimes Y) \xrightarrow{\bT} \aeX(T Y,T(X \btimes Y))\;.$$
\item We define morphisms $\otimes_{XY}$, $\widetilde{\otimes}_{XY}$ as the following composites:
$$\otimes_{XY} := (T X \btimes T Y \xrightarrow{t''_{TX Y}} T(TX \btimes Y) \xrightarrow{T t'_{XY}} TT(X \btimes Y) \xrightarrow{\kappa} T (X \btimes Y))\;,$$
$$\widetilde{\otimes}_{XY} := (T X \btimes T Y \xrightarrow{t'_{X TY}} T (X \btimes T Y) \xrightarrow{T t''_{XY}} T T (X \btimes Y) \xrightarrow{\kappa} T (X \btimes Y))\;.$$
\item $\TT$ is \textit{commutative} if $\otimes_{XY} = \widetilde{\otimes}_{XY}$ for all objects $X,Y$ in $\X$.
\end{enumerate}
\end{DefSub}
\begin{RemSub} \label{rem:comm_inv_under_iso}
Not surprisingly, the property of commutativity is invariant under isomorphism of $\X$-monads \pbref{par:adj_and_mnd}, as one readily checks.
\end{RemSub}

\begin{ThmSub}\label{thm:comm_sm_mnd}\textnormal{(Kock \cite{Kock:SmComm})}
Let $\TT = (T,\eta,\mu)$ be an ordinary monad on $\X$.  Then there is a bijection between the following kinds of structure on $\TT$:
\begin{enumerate}
\item $\X$-enrichments of $T$ making $\TT$ a commutative $\X$-monad on $\aeX$;
\item monoidal structures on $T$ making $\TT$ a symmetric monoidal monad.
\end{enumerate}
In particular, if $\TT$ is equipped with the structure of a symmetric monoidal monad, then the associated $\X$-enrichment of $T$ consists of the structure morphisms
$$\aeX(X,Y) \rightarrow \aeX(TX,TY)$$
gotten as the transposes of the following composites:
\begin{equation}\label{eqn:transp_tc}TX \btimes \aeX(X,Y) \xrightarrow{1 \btimes \eta} TX \btimes T\aeX(X,Y) \rightarrow T(X \btimes \aeX(X,Y)) \xrightarrow{T(\Ev)} TY\;.\end{equation}
\end{ThmSub}

\begin{DefSub}
Given a symmetric monoidal monad $\TT = (T,\eta,\mu)$ on $\X$, let $\TT^c = (\bT^c,\eta,\mu)$ denote the associated commutative $\X$-monad on $\aeX$ \pbref{thm:comm_sm_mnd}.
\end{DefSub}

\begin{PropSub} \label{thm:tc_via_tgrave}
Given a symmetric monoidal monad $\TT = (T,\eta,\mu)$ on $\X$, the associated $\X$-functor $\bT^c:\aeX \rightarrow \aeX$ is the composite
$$\aeX \xrightarrow{(\eta_*)_{\aeX}} T_*\aeX \xrightarrow{\grave{T}} \aeX\;,$$
where $\eta_*:1_{\XCAT} = (1_{\X})_* \rightarrow T_*$ is gotten by applying $(-)_*$ \pbref{par:ch_base} to the symmetric monoidal transformation $\eta$.
\end{PropSub}
\begin{proof}
$(\eta_*)_{\aeX}$ is identity-on-objects, and $\grave{T}$ acts as $T$ on objects, so both $\bT^c$ and the indicated composite are given as $T$ on objects.  For all $X,Y \in \aeX$ the associated structure morphism of the composite $\X$-functor $\grave{T} \circ (\eta_*)_{\aeX}$ is the composite 
$$\aeX(X,Y) \xrightarrow{\eta} T\aeX(X,Y) \xrightarrow{\grave{T}_{X Y}} \aeX(TX,TY)\;,$$
whose transpose $TX \btimes \aeX(X,Y) \rightarrow TY$ is (by the definition of $\grave{T}$, \bref{par:ch_base}) exactly the composite \eqref{eqn:transp_tc} employed in defining $\bT^c_{X Y}$.
\end{proof}

\begin{PropSub} \label{thm:enradj_assoc_smadj_ind_commmnd}
Let $\xymatrix {\X \ar@/_0.5pc/[rr]_F^(0.4){\eta}^(0.6){\epsilon}^{\top} & & \W \ar@/_0.5pc/[ll]_G}$ be a symmetric monoidal adjunction with $\X$,$\W$ symmetric monoidal closed, and let $\TT$ be the induced symmetric monoidal monad on $\X$.  Then the associated commutative $\X$-monad $\TT^c$ on $\aeX$ coincides with the $\X$-monad $\TT'$ induced by the $\X$-adjunction
$$
\xymatrix {
\aeX \ar@/_0.5pc/[rr]_{\acute{F}}^(0.4){\eta}^(0.6){\epsilon}^{\top} & & G_*\aeW \ar@/_0.5pc/[ll]_{\grave{G}}
}
$$
gotten via \bref{thm:assoc_enr_adj}.  In particular, $\TT'$ is commutative.
\end{PropSub}
\begin{proof}
Letting $\TT = (T,\eta,\mu)$, we have that $\TT^c = (\bT^c,\eta,\mu)$.  By \bref{thm:assoc_enr_adj}, the underlying ordinary adjunction of $\acute{F} \nsststile{\epsilon}{\eta} \grave{G}$ is $F \nsststile{\epsilon}{\eta} G$, so the underlying ordinary monad of $\TT'$ is equal to that of $\TT$; hence $\TT' = (\bT',\eta,\mu)$, where $\bT' = \grave{G}\acute{F}$.  Further, $\bT' = \bT^c$, since the following diagram commutes
$$
\xymatrix@R=4ex {
{\aeX} \ar[rr]^{\bT^c} \ar[dr]|{(\eta_*)_{\aeX}} \ar@/_1.5pc/[ddr]_{\acute{F}} &                                                                              & {\aeX} \\
                                                                             & {(GF)_*\aeX = G_*F_*\aeX} \ar[d]|{G_*\grave{F}} \ar[ur]|{\grave{(GF)}}       &        \\
                                                                             & {G_*\aeW} \ar@/_1.5pc/[uur]_{\grave{G}}                                         &
}
$$
by \bref{thm:tc_via_tgrave}, \bref{prop:comp_assoc_enr_fun}, and the definition of $\acute{F}$ \pbref{par:given_smcadj}.
\end{proof}

\section{Symmetric monoidal closed reflections} \label{sec:smc_refl}

\begin{ThmSub}\label{thm:day_refl}\textnormal{(Day \cite{Day:Refl})}
Let $\B$ be a symmetric monoidal closed category, let $\C$ be a full, replete reflective subcategory of $\B$, with associated adjunction $K \nsststile{}{\rho} J : \C \hookrightarrow \B$, and suppose that
$$\forall B \in \B, C \in \C \;:\; \aeB(B,C) \in \C\;.$$
Then the given adjunction acquires the structure of a symmetric monoidal adjunction, with $\C$ a \textit{closed} symmetric monoidal category.
\end{ThmSub}
\begin{proof}
This follows from \cite{Day:Refl}, 1.2.  Letting $s$ be the symmetry of $\B$, one checks that the property of symmetry of the monoidal functors $K$, $J$ defined in \cite{Day:Refl} follows immediately from the naturality of $s$ and $\rho$.
\end{proof}

\begin{CorSub} \label{thm:enr_refl_closed}
Let $\V$ be a symmetric monoidal closed category, let \hbox{$\bK \nsststile{\sigma}{\rho} \bJ : \eC \hookrightarrow \aeV$} be a $\V$-adjunction, where $\bJ$ is the inclusion of a full, replete sub-$\V$-category.  Then the underlying ordinary adjunction $K \nsststile{\sigma}{\rho} J : \C \hookrightarrow \V$ acquires the structure of a symmetric monoidal adjunction, with $\C$ a \textit{closed} symmetric monoidal category.  Further, the $\V$-monad on $\aeV$ induced by $\bK \nsststile{\sigma}{\rho} \bJ$ is isomorphic to the $\V$-monad induced by the $\V$-adjunction $\acute{K} \nsststile{\sigma}{\rho} \grave{J} : J_*(\aeC) \rightarrow \aeV$ \pbref{thm:assoc_enr_adj}.
\end{CorSub}
\begin{proof}
$\eC$ is closed under $\V$-enriched weighted limits in $\aeV$, so for all $V \in \V$ and $C \in \C$, since $\aeV(V,C)$ is a cotensor $[V,C]$ in $\aeV$, this object lies in $\C$.  Hence \bref{thm:day_refl} applies.

Note that both $\V$-adjunctions in question have underlying ordinary adjunction $K \nsststile{\sigma}{\rho} J$.  Since $J$ is fully faithful, $\sigma$ is an isomorphism.  But $\sigma$ serves as the counit of the $\V$-adjunction $\acute{K} \nsststile{\sigma}{\rho}\grave{J}$, so $\grave{J}$ is a fully faithful $\V$-functor (by, e.g., \cite{Ke:Ba}, 1.11).  But $\grave{J}$ has underlying ordinary functor $J$, so $\grave{J}$ is injective on objects and has image exactly $\eC$, and therefore $\grave{J}$ factors through an isomorphism $\bP $ such that
$$
\xymatrix@R=0.5ex {
{J_*(\aeC)} \ar[dr]_{\grave{J}} \ar[rr]^{\bP}_{\sim} &                     & \eC \ar@{^{(}->}[dl]^{\bJ} \\
                                                     & \aeV                & 
}
$$
commutes.  It is now straightforward to obtain the needed isomorphism of $\V$-monads by using \bref{thm:uniq_adj}.
\end{proof}

\section{Enriched orthogonality and finite well-completeness} \label{sec:enr_orth_fwc}

In Day's paper \cite{Day:AdjFactn} (whose results we shall employ in \bref{sec:enr_adj_factn})  enriched notions of orthogonality, factorization systems, and completeness are employed, and it is the purpose of this section to examine certain aspects of their relation to the corresponding ordinary notions.  Let $\V$ be a locally small symmetric monoidal closed category.

\begin{DefSub} \label{def:enr_notions}
Let $\eA$ be a $\V$-category.
\begin{enumerate}
\item A morphism $m:B_1 \rightarrow B_2$ in $\eA$ (i.e., in $\A$) is a \textit{$\V$-mono} if $\eA(A,m):\eA(A,B_1) \rightarrow \eA(A,B_2)$ is a monomorphism in $\V$ for every object $A$ of $\eA$.  A morphism $e$ in $\eA$ is a \textit{$\V$-epi} if $f$ is a $\V$-mono in $\eA^\op$.

\item $\Mono_\V\eA$ and $\Epi_\V\eA$ are the classes of all $\V$-monos and $\V$-epis, respectively, in $\eA$.

\item For morphisms $e:A_1 \rightarrow A_2$, $m:B_1 \rightarrow B_2$ in $\eA$ we say that $e$ is \textit{$\V$-orthogonal} to $m$, written $e \downarrow_\V m$, if the commutative square
\begin{equation}\label{eqn:orth_pb}
\xymatrix@R=2ex {
\eA(A_2,B_1) \ar[rr]^{\eA(A_2,m)} \ar[d]_{\eA(e,B_1)}  & & \eA(A_2,B_2) \ar[d]^{\eA(e,B_2)} \\
\eA(A_1,B_1)  \ar[rr]^{\eA(A_1,m)}                     & & \eA(A_1,B_2)
}
\end{equation}
is a pullback in $\V$.

\item Given classes $\E$, $\M$ of morphisms in $\eA$, we define
$$\E^{\downarrow_\V} := \{m \in \mor\eA \;|\; \forall e \in \E \;:\; e \downarrow_\V m\}\;,$$
$$\M^{\uparrow_\V} := \{e \in \mor\eA \;|\; \forall m \in \E \;:\; e \downarrow_\V m\}\;.$$

\item A \textit{$\V$-prefactorization-system} on $\eA$ is a pair $(\E,\M)$ of classes of morphisms in $\eA$ such that $\E^{\downarrow_\V} = \M$ and $\M^{\uparrow_\V} = \E$.  A \textit{$\V$-factorization-system} is a $\V$-prefactorization-system such that \textit{(\E,\M)-factorizations exist} --- i.e., every morphism in $\eA$ factors as a morphism in $\E$ followed by a morphism in $\M$.  A $\V$-prefactorization system $(\E,\M)$ is said to \textit{$\V$-proper} if $\E \subs \Epi_\V\eA$ and $\M \subs \Mono_\V\eA$.

\item The morphisms in the class $\StrMono_\V\eA := (\Epi_\V\eA)^{\downarrow_\V} \cap \Mono_\V\eA$ are called \textit{$\V$-strong-monomorphisms}.

\item A \textit{$\V$-limit} of an ordinary functor $D:\J \rightarrow \A$ is a limit of $D$ that is preserved by each ordinary functor $\eA(A,-):\A \rightarrow \V$.  As special cases of $\V$-limits we define \textit{$\V$-products, $\V$-pullbacks, $\V$-fiber-products}, etc.  \textit{$\V$-colimits} are defined as $\V$-limits in $\A^\op$.
\end{enumerate}
\end{DefSub}
\begin{RemSub} \label{rem:enr_notions}
In the above definitions \pbref{def:enr_notions}, one obtains the corresponding ordinary notion (for locally small categories) by taking $\V := \Set$, and in this case we often omit the explicit indication of $\V$.

$\V$-limits coincide with the \textit{conical limits} of \cite{Ke:Ba}.  Note that $\V$-orthogonality in $\eA$ implies orthogonality in $\A$, and every $\V$-mono (resp. $\V$-epi, $\V$-limit) in $\eA$ is a mono (resp. epi, limit) in $\A$.
\end{RemSub}

\begin{PropSub} \label{thm:lim_mono_epi_in_base_are_enr}
Any ordinary mono (resp. epi, limit, colimit) in $\aeV$ is a $\V$-mono (resp. $\V$-epi, $\V$-limit, $\V$-colimit).  Hence $\Mono_{\V}\aeV = \Mono\V$ and $\Epi_{\V}\aeV = \Epi\V$.
\end{PropSub}
\begin{proof}
Regarding limits and monos, each ordinary functor $\aeV(V,-):\V \rightarrow \V$ is right adjoint and hence preserves limits and monos.  Regarding epis and colimits, each functor $\aeV^\op(V,-) = \aeV(-,V) : \V^\op \rightarrow \V$ is right adjoint (to $\aeV(-,V)$) and hence sends monos (resp. limits) in $\V^\op$ (i.e. epis, resp. colimits, in $\V$) to monos (resp. limits).
\end{proof}

\begin{PropSub} \label{thm:crit_factn_sys}
Let $\E$, $\M$ be classes of morphisms in a $\V$-category $\eA$.  Suppose that (i) each of $\E$ and $\M$ is closed under composition and contains all isomorphisms, (ii) $\E \subs \M^{\uparrow_\V}$, and (iii) $(\E,\M)$-factorizations exist.  Then $(\E,\M)$ is a $\V$-factorization system on $\eA$.
\end{PropSub}
\begin{proof}
We have
\begin{equation}\label{eqn:orth_incls}\E \subs \M^{\uparrow_\V} \subs \M^\uparrow\;,\;\;\;\;\M \subs \E^{\downarrow_\V}\subs \E^\downarrow\;.\end{equation}  
Hence by \cite{FrKe} 2.2 (and 2.2.2 in particular), $(\E,\M)$ is an ordinary factorization system on $\A$, so $\E = \M^\uparrow$ and $\M = \E^\downarrow$ and hence (by \eqref{eqn:orth_incls}) $\E = \M^{\uparrow_\V}$ and $\M = \E^{\downarrow_\V}$.
\end{proof}

\begin{CorSub} \label{thm:vfactn_sys_is_factn_sys}
Every $\V$-factorization system $(\E,\M)$ on $\eA$ is an ordinary factorization system on $\A$.
\end{CorSub}
\begin{proof}
Since $\E \subs \M^{\uparrow_\V} \subs \M^\uparrow$, we may invoke \bref{thm:crit_factn_sys} with regard to the ordinary category $\A$.
\end{proof}

A statement of the following proposition appears in an entry on the collaborative web site \textit{Nlab}, at \href{http://ncatlab.org/nlab/revision/enriched+factorization+system/2}{http://ncatlab.org/nlab/revision/enriched+factorization+system/2}:
\begin{PropSub} \label{thm:enr_orth_from_ord}
Let $\E$ be a class of morphisms in a $\V$-category $\eA$.  Suppose that $\eA$ is tensored, and suppose that for each object $V$ in $\V$, $\E$ is stable under the application of $V \otimes (-):\eA \rightarrow \eA$.  Then $\E^{\downarrow_\V} = \E^{\downarrow}$.
\end{PropSub}
\begin{proof}
Enriched orthogonality implies ordinary, so it suffices to show that $\E^\downarrow \subs \E^{\downarrow_\V}$.  Letting $m:B_1 \rightarrow B_2$ lie in $\E^\downarrow$ and $e:A_1 \rightarrow A_2$ lie in $\E$, we must show that $e \downarrow_\V m$.  It suffices to show that each functor $\V(V,-):\V \rightarrow \Set$ ($V \in \V$) sends the square \eqref{eqn:orth_pb} to a pullback square in $\Set$.  Since $\eA$ is tensored, we have in particular that
$$\V(V,\eA(A,B)) \cong \A(V \otimes A,B)$$
naturally in $A,B \in \A$.  The diagram of sets obtained by applying $\V(V,-)$ to the square \eqref{eqn:orth_pb} is therefore isomorphic to the following diagram
$$
\xymatrix@R=2ex {
\A(V \otimes A_2,B_1) \ar[rr]^{\A(V \otimes A_2,m)} \ar[d]_{\A(V \otimes e,B_1)}  && \A(V \otimes A_2,B_2) \ar[d]^{\A(V \otimes e,B_2)} \\
\A(V \otimes A_1,B_1)  \ar[rr]^{\A(V \otimes A_1,m)}                     && {\A(V \otimes A_1,B_2)\;,}
}
$$
which is a pullback in $\Set$ since $V \otimes e \in \E$ and hence $V \otimes e \downarrow m$.
\end{proof}

\begin{PropSub} \label{thm:ord_str_mono_in_base_is_enriched}
$\StrMono_{\V}\aeV = \StrMono\V$.
\end{PropSub}
\begin{proof}
Since epis in $\V$ are preserved by each left adjoint functor $V \otimes (-):\V \rightarrow \V$, the class $\E := \Epi\V$ satisfies the hypotheses of \bref{thm:enr_orth_from_ord}, and we compute (also using \bref{thm:lim_mono_epi_in_base_are_enr}) that
$$(\Epi_\V\aeV)^{\downarrow_\V} = (\Epi\V)^{\downarrow_\V} = (\Epi\V)^{\downarrow}$$
and hence $\StrMono_{\V}\aeV = (\Epi_\V\aeV)^{\downarrow_\V} \cap \Mono_\V\aeV = (\Epi\V)^{\downarrow} \cap \Mono\V = \StrMono\V$.
\end{proof}

The following notion is a $\V$-enriched generalization of the similarly named notion in \cite{CHK}:
\begin{DefSub} \label{def:enr_fwc}
A $\V$-category $\eA$ is \textit{$\V$-finitely-well-complete ($\V$-f.w.c.)} if
\begin{enumerate}
\item $\eA$ has all finite $\V$-limits, and
\item for every (class-indexed) family of $\V$-strong-monos $(M_i \rightarrowtail A)$ in $\eA$ with common codomain, there is a $\V$-fiber-product $M \rightarrow A$ that is again a $\V$-strong-mono.
\end{enumerate}
\end{DefSub}
\begin{RemSub} \label{rem:fwc}
In the case that $\V = \Set$, we say that the ordinary category $\A$ is \textit{finitely well-complete}.  Under the assumption of condition $1.$ we have  by \cite{CHK}, pg. 292, that $(\Epi\A,\StrMono\A)$ is a prefactorization system on $\A$, so by \cite{FrKe}, 2.1.1, $\StrMono\A$ is closed under fiber products in $\A$.  Hence in this case we may replace 2. by
\begin{enumerate}
\item[$2'$.] every (class-indexed) family of strong monos in $\A$ has a fiber product in $\A$.
\end{enumerate}
\end{RemSub}

\begin{PropSub} \label{thm:base_fwc}
Suppose $\V$ is finitely well-complete.  Then $\aeV$ is $\V$-finitely-well-complete, and $(\Epi_\V\aeV,\StrMono_\V\aeV) = (\Epi\V,\StrMono\V)$ is a $\V$-proper $\V$-factor\-ization-system on $\aeV$.
\end{PropSub}
\begin{proof}
Since finite $\V$-limits, $\V$-fiber-products, $\V$-epis, and $\V$-strong-monos in $\aeV$ are the same as the corresponding ordinary notions in $\V$ (by \bref{thm:lim_mono_epi_in_base_are_enr}, \bref{thm:ord_str_mono_in_base_is_enriched}), we find that $\aeV$ is $\V$-f.w.c. and the two pairs given are equal.  Also, by \cite{CHK}, 3.2, $(\E,\M) := (\Epi\V,\StrMono\V)$ is a factorization system on $\V$.  By the definition of $\StrMono_\V\aeV$, we have that $\M \subs \E^{\downarrow_\V}$, and we thus have also that $\E \subs \M^{\uparrow_\V}$.  Further,
$$\E^{\downarrow_\V} \subs \E^{\downarrow} = \M\;,\;\;\;\;\M^{\uparrow_\V} \subs \M^{\uparrow} = \E$$
since $(\E,\M)$ is a prefactorization system.
\end{proof}

\section{Enriched adjoint functor factorization} \label{sec:enr_adj_factn}

Let $\V$ be a symmetric monoidal closed category.  The following is the basic lemma that enables one to factorize a $\V$-adjunction through a reflective sub-$\V$-category:

\begin{LemSub} \label{thm:adj_factn_lemma}
Given a $\V$-adjunctions $\bS \nsststile{\epsilon}{\eta} \bT : \eC \rightarrow \eB$ and $\bK \nsststile{}{\rho} \bJ : \eB' \hookrightarrow \eB$, where $\bJ$ is the inclusion of a full sub-$\V$-category, let us suppose that the image of $\bT$ lies in $\eB'$.  Then there is a $\V$-adjunction $\bS' \nsststile{\epsilon'}{\eta'} \bT':\eC \rightarrow \eB'$ with $\bJ \bT' = \bT$, $\bS' = \bS \bJ$, $\bJ\eta' = \eta\bJ$, and $\epsilon' = \epsilon$.
\end{LemSub}
\begin{proof}
$T'$ is just the corestriction of $T$, the components of $\eta'$ are just those of $\eta$, and the triangular equations are readily verified.
\end{proof}

\begin{DefSub} \label{def:orth_subcat_sigma}
\begin{enumerate}
\item For a morphism $f:A_1 \rightarrow A_2$ in a $\V$-category $\eA$ and an object $B$ in $\eA$, we say that $f$ is \textit{$\V$-orthogonal} to $B$, written $f \bot_\V B$, if $\eA(f,B):\eA(A_2,B) \rightarrow \eA(A_1,B)$ is an isomorphism in $\V$.  (If $\eA$ has a $\V$-terminal object $1$, then it is easy to show that $f \bot_\V B$ iff $f \downarrow_\V !_B$, where $!_B:B \rightarrow 1$.)
\item Given a class of morphisms $\sS$ in $\eA$, we let $\eA_\sS$ be the full sub-$\V$-category of $\eA$ determined by those objects of $\eA$ to which every morphism in $\sS$ is $\V$-orthogonal.
\item Given a functor $S:\C \rightarrow \D$, let $\Sigma_S$ denote the class of all morphisms in $\C$ inverted by $S$ (i.e. sent to an isomorphism in $\D$).
\end{enumerate}
\end{DefSub}

\begin{LemSub} \label{thm:orth_of_morphs_in_orth_subcat}
Let $\eA$ be a $\V$-category and $\sS$ a class of morphism in $\eA$.  Then for any morphisms $e:A_1 \rightarrow A_2$ in $\sS$ and $m:B_1 \rightarrow B_2$ in $\eA_{\sS}$, we have that $e \downarrow_\V m$ in $\eA$.
\end{LemSub}
\begin{proof}
Since $e \bot_\V B_1$ and $e \bot_\V B_2$, the left and right sides of the commutative square \eqref{eqn:orth_pb} are isomorphisms, so the square is a pullback.
\end{proof}

We employ results of Day \cite{Day:AdjFactn} in proving the following:

\begin{ThmSub} \label{thm:adj_factn}
Let $\xymatrix {\eB \ar@/_0.5pc/[rr]_\bS^(0.4){\eta}^(0.6){\epsilon}^{\top} & & \eC \ar@/_0.5pc/[ll]_\bT}$ be a $\V$-adjunction, and assume that $\V$ is locally small.  Suppose that $(\E,\M) := (\Epi_\V\eB,\StrMono_\V\eB)$ is a $\V$-proper $\V$-factorization-system on $\eB$.  Suppose that $\eB$ is cotensored and has $\V$-equalizers, $\V$-pullbacks of $\M$-morphisms, and arbitrary $\V$-fiber-products of $\M$-morphisms.  Then there are $\V$-adjunctions
$$
\xymatrix {
\eB \ar@/_0.5pc/[rr]_{\bK}^(0.4){\rho}^(0.6){}^{\top} & & {\eB''} \ar@{_{(}->}@/_0.5pc/[ll]_{\bJ} \ar@/_0.5pc/[rr]_{\bS''}^(0.4){\eta''}^(0.6){}^{\top} & & {\eC} \ar@/_0.5pc/[ll]_{\bT''}
}
$$
with $\bJ\bT''  = \bT$, $\bS''  = \bS\bJ$, $\bJ$ a full sub-$\V$-category inclusion, and $\bS''$ conservative (i.e. $\bS'' $ reflects isomorphisms).  Further, each component of $\eta''$ lies in $\M$, and $\eB'' = \eB_{\Sigma_S}$.
\end{ThmSub}
\begin{proof}
By \bref{thm:vfactn_sys_is_factn_sys}, $(\E,\M)$ a proper (ordinary) factorization system on $\B$, so by (the duals of) \cite{FrKe} 2.1.1 and 2.1.3, $\M$ is closed under pullback and under fiber products and contains all equalizers.  Since every $\V$-limit is an ordinary limit, $\M$ is closed under $\V$-pullback and under $\V$-fiber products and contains all $\V$-equalizers.  Hence, in the terminology of \cite{Day:AdjFactn}, $\B$ is \textit{$\M$-complete}.

Let $\bJ':\eB' \hookrightarrow \eB$ be inclusion of the full sub-$\V$-category of $\eB$ consisting of those objects $B$ for which $\eta_B \in \M$.  It is shown in \cite{Day:AdjFactn}, 1.2, that there is a $\V$-adjunction $\bK' \nsststile{}{\rho'} \bJ'$ such that $\rho'$ is inverted by $\bS$.

Observe that for each $C \in \eC$, $TC$ lies in $\B'$, since $\eta_{TC}$ is a split mono and hence an equalizer and therefore lies in $\M$.  Therefore we may use \bref{thm:adj_factn_lemma} to obtain a $\V$-adjunction $\bS' \nsststile{\epsilon}{\eta'} \bT' : \eC \rightarrow \eB'$ with $\bJ' \bT' = \bT$, $\bS' = \bS \bJ'$ and $\bJ'\eta' = \eta\bJ'$.

Let $\E'$ and $\M'$ be the intersections of $\mor \B'$ with $\E$ and $\M$, respectively.  By \cite{Day:AdjFactn} 1.1, every $\M$-subobject of an object of $\eB'$ again lies in $\eB'$, so since $\eB$ has $(\E,\M)$-factorizations, $\eB'$ has $(\E',\M')$-factorizations.  It follows that $(\E',\M')$ satisfies the hypotheses of \bref{thm:crit_factn_sys} and hence is a $\V$-factorization-system on $\eB'$.  Further, since $\eB'$ is a $\V$-reflective sub-$\V$-category of the cotensored $\M$-complete $\V$-category $\eB$, we deduce also that $\eB'$ is cotensored and $\M'$-complete.

Therefore the $\V$-adjunction $\bS' \nsststile{\epsilon}{\eta'} \bT'$ satisfies the hypotheses of 2.3 of \cite{Day:AdjFactn}, whence we obtain a $\V$-adjunction $\bK'' \nsststile{}{\rho''} \bJ'' : \eB'' \hookrightarrow \eB'$, where $\bJ''$ is the inclusion of the full sub-$\V$-category $\eB'' := \eB'_{\Sigma_{S'}}$ of $\eB'$.

For each $C \in \eC$, $T' C$ lies in $\eB''$, since for each morphism $h:B_1' \rightarrow B_2'$ in $\Sigma_{S'}$, we have a commutative square
$$
\xymatrix@R=2ex {
{\eB'(B_2',T'C)} \ar[rr]^{\eB'(h,T'C)} \ar[d]^{\sim} & & {\eB'(B_1',T'C)} \ar[d]^{\sim} \\
{\eC(S'B_2',C)} \ar[rr]_{\eC(S'h,C)}                 & & {\eC(S'B_1',C)}
}
$$
whose right, left, and bottom sides are isomorphisms, so that the top side is an isomorphism.

Hence we may again apply \bref{thm:adj_factn_lemma} to obtain a $\V$-adjunction $\bS'' \nsststile{\epsilon}{\eta''} \bT'' : \eC \rightarrow \eB''$ with $\bJ'' \bT'' = \bT'$, $\bS'' = \bS' \bJ''$ and $\bJ''\eta'' = \eta'\bJ''$.

Composing the $\V$-adjunctions $\bK' \dashv \bJ'$ and $\bK'' \dashv \bJ''$, we obtain a $\V$-adjunction  $\bK \dashv \bJ : \eB'' \hookrightarrow \eB$, and one checks that $\bJ\bT'' = \bT$ and $\bS'' = \bS\bJ$ as needed.

We now show that $\eB'' = \eB_{\Sigma_S}$.  First suppose that $B \in \eB_{\Sigma_S}$.  Then since $S$ inverts $\rho'_B$ we have that $\rho'_B \bot_\V B$, and hence $\B(\rho'_B,B):\B(J'K'B,B) \rightarrow \B(B,B)$ is a bijection.  Hence $1_B$ lies in the image of this map, so $\rho'_B$ is a split mono and therefore lies in $\M$.  But in \cite{Day:AdjFactn}, 1.2, $\rho'_B$ is obtained as the first factor of an $(\E,\M)$-factorization $\eta_B = m \cdot \rho'_B$ of $\eta_B$, where $m \in \M$, whence $\eta_B \in \M$, showing that $B \in \eB'$.  Since $B \in \eB_{\Sigma_S}$, it now follows that $B \in \eB'_{\Sigma_{S'}} = \eB''$.  

Conversely, let $B'' \in \eB''$.  Suppose that $f:B_1 \rightarrow B_2$ lies in $\Sigma_S$.  We have a commutative square
\begin{equation}\label{eq:square_a}
\xymatrix@R=2ex {
{\eB(B_2,B'')} \ar[rr]^{\eB(f,B'')} \ar[d]^{\sim} & & {\eB(B_1,B'')} \ar[d]^{\sim} \\
{\eB'(K'B_2,B'')} \ar[rr]^{\eB'(K'f,B'')}        & & {\eB'(K'B_1,B'')}
}
\end{equation}
whose right and left sides are isos.  But $K'f \in \Sigma_{S'}$, since we have a commutative square 
$$
\xymatrix@R=2ex {
{SB_1}                   \ar[rr]^{Sf}_{\sim} \ar[d]^{\sim}_{S\rho'_{B_1}}    &    & {SB_2} \ar[d]_{\sim}^{S\rho'_{B_2}}\\
{SJ'K'B_1 = S'K'B_1} \ar[rr]_{SJ'K'f = S'K'f}         &    & {SJ'K'B_2 = S'K'B_2}
}
$$
whose left, right, and top sides are isos, whence $S'K'f$ is iso.  Hence since $B'' \in \eB'_{\Sigma_{S'}}$, the bottom side of \eqref{eq:square_a} is an iso, so the top side of \eqref{eq:square_a} is an iso, whence $f \bot_\V B''$.  Hence $B'' \in \eB_{\Sigma_S}$.

It now follows that $S''$ is conservative, since if a morphism $f:B_1'' \rightarrow B_2''$ in $\eB'' = \eB_{\Sigma_S}$ is such that $S''f$ is an isomorphism in $\eC$, then since $S''f$ is simply $Sf$ we have that $f \in \Sigma_S$, so by \bref{thm:orth_of_morphs_in_orth_subcat} we deduce that $f \downarrow_{\V} f$ in $\eB$, from which it follows that $f$ is an isomorphism in $\eB$ and hence in $\eB''$.
\end{proof}

\begin{RemSub} \label{rem:factn_inversion}
In \bref{thm:adj_factn}, $\bK$ inverts the same morphisms as $\bS$ (i.e., $\Sigma_{K} = \Sigma_{S}$), since $\bS  \cong \bS'' \bK$ and $\bS'' $ is conservative.
\end{RemSub}

\begin{CorSub} \label{thm:base_adj_fact}
Suppose that $\V$ is finitely well-complete and locally small.  Then any $\V$-adjunction $\xymatrix {\aeV \ar@/_0.5pc/[rr]_\bS^(0.4){}^(0.6){}^{\top} & & \eC \ar@/_0.5pc/[ll]_\bT}$ factors as in \bref{thm:adj_factn} (when we set $\eB := \aeV$).
\end{CorSub}
\begin{proof}
An invocation of \bref{thm:base_fwc} shows that the hypotheses of \bref{thm:adj_factn} are satisfied.
\end{proof}

\vspace{1pc}
\begin{center}\normalsize{{P\textsc{art} II:  N\textsc{atural} D\textsc{istributions} \textsc{and} F\textsc{ubini}}}\end{center}

\section{The natural distribution monad and Fubini}

\begin{ParSub} \label{par:smc_adj}
Throughout the sections that follow, we will consider a given symmetric monoidal adjunction $F \nsststile{\epsilon}{\eta} G:\sL \rightarrow \X$, as in \eqref{eqn:smcadj}, where $\X = (\X,\btimes,I)$ and $\sL = (\sL,\otimes,R)$ are \textit{closed} symmetric monoidal categories.
By \bref{thm:assoc_enr_adj}, there is an associated $\X$-enriched adjunction $\acute{F} \nsststile{\epsilon}{\eta} \grave{G}:G_*\aeL \rightarrow \aeX$ \eqref{eqn:assoc_enr_adj} whose underlying ordinary adjunction coincides with that of $F \nsststile{\epsilon}{\eta} G$.
\end{ParSub}

We are chiefly concerned with those cases in which $\X$ is a cartesian closed category, whose objects are thought of as `spaces' of some sort, and $\sL$ is the category of $R$-modules (or $R$-vector-space objects) in $\X$ for some commutative ring object $R$ in $\X$.  In particular, our principal example is as follows:

\begin{PropSub} \label{thm:conv_smadj}
Let $\X := \Conv$ and $\sL := \RVect(\X)$ (where $R := \RR$ or $\CC$) be the categories of convergence spaces and convergence vector spaces, respectively.  Then there is a symmetric monoidal adjunction $F \nsststile{\epsilon}{\eta} G : \sL \rightarrow \X$, with $\X$ cartesian closed and $\sL$ symmetric monoidal closed, where $G$ is the forgetful functor and $\sL$ has unit object $R$.
\end{PropSub}
\begin{proof}
Firstly, $\X$ is cartesian closed; see, e.g., \cite{BeBu} 1.5.2.  We may apply 4.6 of \cite{Seal:Tensor} in order to obtain an adjunction $F \nsststile{\epsilon}{\eta} G : \sL \rightarrow \X$ with $\sL = (\sL,\otimes,R)$ symmetric monoidal closed and $F$ a strong monoidal functor.  Moreover, it is noted in \cite{Seal:Tensor} that for each $X \in \X$, the underlying $R$-module of $FX$ is simply the usual free $R$-module on the underlying set of $X$.  Further, for all $E_1, E_2 \in \sL$, the underlying $R$-module of the monoidal product $E_1 \otimes E_2$ in $\sL$ is the usual tensor product of $R$-modules.  Moreover, for all $X,Y \in \X$, the structure isomorphism $FX \otimes FY \rightarrow F(X \times Y)$ has as its inverse the unique linear map $F(X \times Y) \rightarrow FX \otimes FY$ given on generators $(x,y) \in X \times Y$ by $(x,y) \mapsto x \otimes y$.  Further, each symmetry isomorphism $E_1 \otimes E_2 \rightarrow E_2 \otimes E_1$ is the unique linear map given on generators by $e_1 \otimes e_2 \mapsto e_2 \otimes e_1$.  Using these facts, the symmetry law for the monoidal functor $F$ is immediately verified.  Hence $F$ is a strong symmetric monoidal functor, and the needed symmetric monoidal adjunction is obtained by \cite{Ke:Doctr}, 1.5.
\end{proof}

\begin{DefSub} \label{def:nat_dist}
Given data as in \bref{par:smc_adj}, the $\X$-category $\eL := G_*\aeL$ is cotensored (by \bref{thm:assoc_enr_adj}), so we have in particular an $\X$-adjunction
\begin{equation}\label{eq:hom_cotensor}
\xymatrix {
\aeX \ar@/_0.5pc/[rr]_{[-,R]}^(0.4){\delta}^(0.6){\sigma}^{\top} & & {\eL^\op\;.} \ar@/_0.5pc/[ll]_{\eL(-,R)}
}
\end{equation}
We call the $\X$-monad $\DD = (\bD,\delta,\kappa)$ on $\aeX$ induced by this $\X$-adjunction the \textit{natural distribution monad}.  Hence 
$$\bD X = \eL([X,R],R)\;\;\;\;\text{$\X$-naturally in $X \in \X$.}$$
\end{DefSub}

\begin{ExaSub}\label{exa:nat_dist_conv}
In our principal example \pbref{thm:conv_smadj}, we can form cotensors $[X,E]$ in $\eL$ (where $X \in \X$ and $E \in \eL$) by equipping the space $\aeX(X,GE)$ of continuous $E$-valued functions with the pointwise vector space structure.  The associated cotensor unit morphism $\delta_X:X \rightarrow \eL([X,E],E)$ sends each $x \in X$ to the \textit{Dirac functional} $\lambda f.f(x)$.  For $X \in \aeX$, $DX = \eL([X,R],R)$ is the convergence vector space of continuous linear functionals $[X,R] \rightarrow R$.  In the case that $X$ is a compact Hausdorff topological space, it is well-known that the convergence structure on $G[X,R] = \aeX(X,R)$ (namely, \textit{continuous convergence}) coincides with \textit{uniform convergence} (induced by the $\infty$-norm); see, e.g., \cite{McCNta} III.1.  Hence the elements of $DX$ are the continuous linear functionals on the Banach space of $R$-valued continuous functions on $X$ --- i.e. the $R$-valued Radon measures on $X$.
\end{ExaSub}

The following lemma reduces our task of proving our Fubini theorem for convergence spaces \pbref{thm:fubini_for_conv} to the problem of proving that the natural distribution monad on $\Conv$ is commutative:

\begin{LemSub} \label{thm:comm_is_fubini_for_conv}
The statement of \bref{thm:fubini_for_conv} is equivalent to the statement that the natural distribution monad $\DD$ on $\Conv$ \pbref{exa:nat_dist_conv} is commutative.
\end{LemSub}
\begin{proof}
Let $X,Y$ be convergence spaces.  The morphisms $t'_{XY}:DX \times Y \rightarrow D (X \times Y)$ and $t''_{XY}:X \times D Y \rightarrow D (X \times Y)$ \pbref{def:comm_mnd} are given by
$$t'_{XY}(\mu,y) = \lambda f.\mu(\lambda x.f(x,y))\;,\;\;\;\;t''_{XY}(x,\nu) = \lambda f.\nu(\lambda y.f(x,y))\;.$$
Next note that that the components of the counit $\sigma$ of the adjunction $[-,R] \dashv \eL(-,R)$ inducing $\DD$ \pbref{def:nat_dist} are the morphisms $\sigma_E:E \rightarrow [\eL(E,R),R]$ of $\sL$ given by $\sigma_E(e) = \lambda \phi.\phi(e)$.  Note also that for each $X \in \aeX$, 
$$\kappa_X = \eL(\sigma_{[X,R]},R) : \eL([D X,R],R) \rightarrow \eL([X,R],R)\;.$$
Using these facts, it is straightforward to compute that the maps $\otimes_{XY},\widetilde{\otimes}_{X Y}:D X \times D Y \rightarrow D (X \times Y) = \eL([X \times Y,R],R)$ of \bref{def:comm_mnd} send a pair $(\mu,\nu) \in DX \times DY$ to the functionals $\otimes_{XY}(\mu,\nu),\widetilde{\otimes}_{X Y}(\mu,\nu)$ whose values at each $f \in [X \times Y,R]$ are the left- and right-hand-sides (respectively) of the Fubini equation \eqref{eqn:lambda_fubini} (equivalently, \eqref{eqn:fubini}).
\end{proof}

\section{Statement of the abstract Fubini theorem}

\begin{DefSub} \label{def:dualization}
Given data as in \bref{par:smc_adj}, we make the following definitions:
\begin{enumerate}
\item We call the $\sL$-adjunction 
$$
\xymatrix {
\aeL \ar@/_0.5pc/[rr]_{\aeL(-,R)}^(0.4){}^(0.6){}^{\top} & & {\aeL^\op} \ar@/_0.5pc/[ll]_{\aeL(-,R)}
}
$$
the \textit{dualization $\sL$-adjunction} and $(-)^* := \aeL(-,R)$ \textit{the dualization $\sL$-functor}.
\item The $\sL$-monad $\HH = (\bH ,\partial,\gamma)$ on $\aeL$ induced by the dualization $\sL$-adjunction is called the \textit{double-dualization $\sL$-monad}.
\item An object $E \in \sL$ is said to be \textit{reflexive (in $\aeL$)} if the morphism $\partial_E:E \rightarrow HE = E^{**}$ is an isomorphism in $\sL$.
\end{enumerate}
\end{DefSub}
\begin{ExaSub} \label{exa:refl_conv_vect}
In the setting of convergence vector spaces \pbref{thm:conv_smadj}, the following results of Butzmann \cite{Bu} show that reflexive objects are abundant:
\begin{enumerate}
\item For a locally convex topological vector space $E$ (considered as an object of $\sL$), the double-dual $E^{**}$ (taken in $\aeL$) is the Cauchy-completion of $E$, and $E$ is reflexive in $\aeL$ if and only if $E$ is Cauchy-complete.
\item Every space $\aeX(X,R)$ of continuous $R$-valued functions on a convergence space $X \in \X$ is reflexive in $\aeL$ when endowed with the pointwise $R$-vector-space structure.  As noted in \bref{exa:nat_dist_conv}, each such convergence vector space $\aeX(X,R)$ is a cotensor $[X,R]$ in the $\X$-category $\eL = G_*\aeL$.
\end{enumerate}
\end{ExaSub}

We can now state our main abstract theorem, to be proved in the following sections:
\begin{ThmSub} \label{thm:general_fubini}
Let $\xymatrix {\X \ar@/_0.5pc/[rr]_F^(0.4){\eta}^(0.6){\epsilon}^{\top} & & \sL \ar@/_0.5pc/[ll]_G}$ be a symmetric monoidal adjunction, with $\X$ and $\sL$ symmetric monoidal closed and $\sL$ locally small and finitely well-complete \pbref{rem:fwc}.  Suppose that each cotensor $[X,R]$ in $G_*\aeL$ is reflexive, where $X \in \X$ and $R$ is the unit object of $\sL$.  Then the natural distribution monad $\DD$ is commutative.
\end{ThmSub}
\begin{RemSub} \label{rem:base_fwc}
The hypothesis in \bref{thm:general_fubini} that $\sL$ is finitely well-complete (f.w.c.) may be replaced by the hypotheses that $\X$ is f.w.c. and that $G$ \textit{creates limits} (\cite{MacL}), as we now show.  Indeed, under these alternate hypotheses, $\sL$ is clearly finitely complete.  Further, it is known (e.g. \cite{Wy:QuTo}, 10.5) that right adjoints (such as $G$) preserve strong monomorphisms.  It follows that since $\X$ has arbitrary fiber products of strong monomorphisms and $G$ creates fiber products, $\sL$ has arbitrary fiber products of strong monomorphisms.  Hence, in view of \bref{rem:fwc}, $\sL$ is f.w.c.
\end{RemSub}
\begin{RemSub}
Further to \bref{rem:base_fwc}, any small-complete category $\X$ whose objects each have but a set of strong subobjects is f.w.c.  Hence, in particular, every category $\X$ topological over $\Set$ is f.w.c., since the strong monomorphisms in $\X$ are exactly the \textit{embeddings} or initial injections (e.g. by \cite{Wy:QuTo} 11.9), so that strong subobjects correspond bijectively to subset inclusions.
\end{RemSub}

\begin{CorSub}\label{thm:nat_dist_on_conv_comm}
The natural distribution monad $\DD$ on the category $\X = \Conv$ of convergence spaces \pbref{exa:nat_dist_conv} is commutative.
\end{CorSub}
\begin{proof}
$\DD$ is the natural distribution monad associated to the symmetric monoidal adjunction of \bref{thm:conv_smadj}, which satisfies the hypotheses of \bref{thm:general_fubini}, as follows.
By \bref{exa:refl_conv_vect}, the cotensors $[X,R]$ in $G_*\aeL$ are reflexive in $\aeL$.  Further, $G$ creates limits (by \cite{He} 2.9, since by \cite{Nel:EnrAlg}, 4.4, $G$ is a \textit{regular functor}).  Hence since $\X$ is topological over $\Set$ and hence f.w.c., we deduce that $\sL$ is f.w.c. by \bref{rem:base_fwc}.  Also, $\sL$ is locally small.
\end{proof}
\begin{RemSub}
In view of \bref{thm:comm_is_fubini_for_conv}, our Fubini theorem for convergence spaces \pbref{thm:fubini_for_conv} follows from \bref{thm:nat_dist_on_conv_comm}.
\end{RemSub}

\section{Natural distributions via double-dualization}

The first key observation on our way to the proof of our main theorem \pbref{thm:general_fubini} is as follows.  Again, we work with given data as in \bref{par:smc_adj}.

\begin{PropSub} \label{prop:dist_via_dd}
Suppose that $\sL$ has equalizers.  Then, choosing the cotensors in $G_*\aeL$ as in \bref{thm:assoc_enr_adj}.\textnormal{2}, \bref{par:cotensor_dual_of_free}, the following hold:
\begin{enumerate}
\item $\DD$ is the $\X$-monad on $\aeX$ induced by the composite $\X$-adjunction
\begin{equation}\label{eq:comp_with_em_adj}
\xymatrix {
\aeX \ar@/_0.5pc/[rr]_{\acute{F}}^(0.4){\eta}^(0.6){\epsilon}^{\top} & & G_*\aeL \ar@/_0.5pc/[ll]_{\grave{G}} \ar@/_0.5pc/[rr]_{G_*(\bF^{\HH})}^(0.4){}^(0.6){}^{\top} & & {G_*\aeL^{\HH}\;,} \ar@/_0.5pc/[ll]_{G_*(\bG^{\HH})}
}
\end{equation}
where the rightmost $\X$-adjunction is obtained by applying $G_*:\LCAT \rightarrow \XCAT$ to the Eilenberg-Moore $\sL$-adjunction $\bF^{\HH} \dashv \bG^{\HH}$ for the double-dualization $\sL$-monad $\HH$ \pbref{def:dualization}.
\item $\DD$ may be obtained from $\HH$ by first applying the 2-functor $G_*:\LCAT \rightarrow \XCAT$ and then applying the monoidal functor $[\acute{F},\grave{G}]:\XCAT(G_*\aeL,G_*\aeL) \rightarrow \XCAT(\aeX,\aeX)$ \pbref{thm:mon_func_detd_by_adj}, so that
\begin{equation}\label{eq:dist_via_dd}\DD = [\acute{F},\grave{G}](G_*(\HH))\;.\end{equation}
\end{enumerate}
\end{PropSub}
\begin{proof}
We can form $\aeL^{\HH}$ since $\sL$ has equalizers (\cite{Dub} II.1).  By definition, $\DD$ is induced by the `hom-cotensor' $\X$-adjunction \eqref{eq:hom_cotensor}, and by \bref{par:cotensor_dual_of_free}, this $\X$-adjunction is equal to the composite
$$
\xymatrix {
\aeX \ar@/_0.5pc/[rr]_{\acute{F}}^(0.4){\eta}^(0.6){\epsilon}^{\top} & & G_*\aeL \ar@/_0.5pc/[ll]_{\grave{G}} \ar@/_0.5pc/[rr]_{G_*(\aeL(-,R))}^(0.4){}^(0.6){}^{\top} & & {G_*(\aeL^\op)\;,} \ar@/_0.5pc/[ll]_{G_*(\aeL(-,R))}
}
$$
in which the rightmost $\X$-adjunction is gotten by applying $G_*$ to the dualization $\sL$-adjunction $\aeL(-,R) \dashv \aeL(-,R)$ \pbref{def:dualization}.  The $\sL$-monad on $\aeL$ induced by the latter $\sL$-adjunction is $\HH$, so the $\X$-monad on $G_*\aeL$ induced by the rightmost $\X$-adjunction is $G_*(\HH)$.  Hence $\DD = [\acute{F},\grave{G}](G_*(\HH))$.  But $\HH$ is equally the $\sL$-monad induced by $\bF^{\HH} \dashv \bG^{\HH}$, so $G_*(\HH)$ is induced by $G_*(\bF^{\HH}) \dashv G_*(\bG^{\HH})$ and hence $\DD$ is induced by the composite \eqref{eq:comp_with_em_adj}.
\end{proof}

\section{Completeness and completion of $\sL$-objects} \label{sec:compl}

\begin{ParSub} \label{par:factn_of_em_adj}
Again working with data as given in \bref{par:smc_adj}, we now suppose that $\sL$ is locally small and finitely well-complete \pbref{rem:fwc}.  Under these hypotheses, we may employ \bref{thm:base_adj_fact} in order to factorize the Eilenberg-Moore $\sL$-adjunction $\xymatrix {\aeL \ar@/_0.5pc/[rr]_{\bF^{\HH}}^(0.4){\partial}^(0.6){}^{\top} & & {\aeL^{\HH}} \ar@/_0.5pc/[ll]_{\bG^{\HH}}}$, yielding $\sL$-adjunctions
\begin{equation}\label{eqn:factn_through_compl_adj}
\xymatrix {
\aeL \ar@/_0.5pc/[rr]_\bK^(0.4){\rho}^(0.6){}^{\top} & & {\teL} \ar@{_{(}->}@/_0.5pc/[ll]_\bJ \ar@/_0.5pc/[rr]_\bQ ^(0.4){\partial'}^(0.6){}^{\top} & & {\aeL^{\HH}} \ar@/_0.5pc/[ll]_\bP 
}
\end{equation}
with $\bJ\bP  = \bG^{\HH}$, $\bQ  = \bF^{\HH}\bJ$, $\bJ$ a full sub-$\sL$-category inclusion, and $\bQ$ conservative.  Further, each component of $\partial'$ is a strong monomorphism in $\sL$, and $\teL = \aeL_{\Sigma_{\bF^{\HH}}}$.
\end{ParSub}

\begin{DefSub}
The leftmost $\sL$-adjunction in \eqref{eqn:factn_through_compl_adj} induces an idempotent $\sL$-monad $\tHH$ on $\aeL$.  As indicated in the Introduction (at \eqref{eqn:completion_refl}), $\tL$ and $\tHH$ determine the notions of \textit{(functional) completeness} and \textit{completion} of $\sL$-objects.
\end{DefSub}

\begin{RemSub}
Since $\bG^{\HH}$ is conservative, we have by \bref{rem:factn_inversion} that $\Sigma_{K} = \Sigma_{\bF^{\HH}} = \Sigma_{\bG^{\HH}\bF^{\HH}} = \Sigma_{H}$ in the notation of  \bref{def:orth_subcat_sigma}.3 --- i.e. $\bK$ inverts exactly the same morphisms as $\bH$.
\end{RemSub}

\begin{PropSub} \label{thm:mnd_morph_i}
There is a morphism of $\sL$-monads $i:\tHH \rightarrow \HH$ whose components are strong monos.  In particular, the diagram
$$
\xymatrix@R=0.5ex {
1_\sL \ar[dr]_{\rho} \ar[rr]^{\partial} &                         & \bH  \\
                                        & \tbH \ar@{>->}[ur]_i    &    
}
$$
commutes.
\end{PropSub}
\begin{proof}
The existence of the needed monad morphism $i$ follows from \bref{thm:mnd_morph_from_adj_factn}.  In view of \eqref{eqn:comp_morph_mnds}, we have that the underlying $\sL$-natural transformation of $i$ is a composite
$$\tbH = \bJ\bK \xrightarrow{\bJ \partial' \bK} \bJ\bP\bQ\bK \xrightarrow{\sim} {\bG^{\HH}\bF^{\HH}} = \bH$$
whose second factor is an isomorphism.  But by \bref{par:factn_of_em_adj}, each component of $\bJ \partial'$ is a strong mono, so each component of $\bJ \partial' \bK$ is a strong mono and hence the same is true of $i$.
\end{proof}

\begin{PropSub} \label{thm:compl_sm_adj}
The underlying ordinary adjunction $K \nsststile{}{\rho} J : \tL \hookrightarrow \sL$ of $\bK \nsststile{}{\rho} \bJ$ carries the structure of a symmetric monoidal adjunction, with $\tL$ a closed symmetric monoidal category.  The $\sL$-monad on $\aeL$ induced by the associated $\sL$-adjunction \hbox{$\acute{K} \nsststile{}{\rho} \grave{J} : J_*\left(\aetL\right) \rightarrow \aeL$} \pbref{thm:assoc_enr_adj} is isomorphic to $\tHH$.
\end{PropSub}
\begin{proof}
This follows from \bref{thm:enr_refl_closed}.
\end{proof}

\section{Distributions via completion}

As in \bref{sec:compl}, we work with data as given in \bref{par:smc_adj}, again supposing that $\sL$ is locally small and finitely well-complete \pbref{rem:fwc}.

\begin{ParSub} \label{par:mnd_morph_tdd_to_dd}
Applying the 2-functor $G_*:\LCAT \rightarrow \XCAT$ to the monad morphism $i:\tHH \rightarrow \HH$ \pbref{thm:mnd_morph_i}, we obtain a monad morphism $G_*(i):G_*(\tHH) \rightarrow G_*(\HH)$ in $\XCAT$.  Next, applying the monoidal functor $[\acute{F},\grave{G}]:\XCAT(G_*\aeL,G_*\aeL) \rightarrow \XCAT(\aeX,\aeX)$ \pbref{thm:mon_func_detd_by_adj}, we obtain a morphism of $\X$-monads
\begin{equation}\label{eq:morph_td_to_d}\tDD := [\acute{F},\grave{G}](G_*(\tHH)) \xrightarrow{[\acute{F},\grave{G}](i) = GiF} [\acute{F},\grave{G}](G_*(\HH)) = \DD\;,\end{equation}
where the rightmost equation holds by \eqref{eq:dist_via_dd}.  The underlying ordinary functor of the $\X$-monad $\tDD$ thus defined is therefore $\tD = G\tH F$, whereas $D = GHF$.  In contrast with  \bref{prop:dist_via_dd}.1, $\tDD$ is induced by the composite $\X$-adjunction
\begin{equation}\label{eq:comp_adj_ind_ddt}
\xymatrix {
\aeX \ar@/_0.5pc/[rr]_{\acute{F}}^(0.4){\eta}^(0.6){\epsilon}^{\top} & & G_*\aeL \ar@/_0.5pc/[ll]_{\grave{G}} \ar@/_0.5pc/[rr]_{G_*(\bK)}^(0.4){}^(0.6){}^{\top} & & {G_*\teL\;.} \ar@/_0.5pc/[ll]_{G_*(\bJ)}
}
\end{equation}
\end{ParSub}

\begin{PropSub}\label{thm:td_comm}
The $\X$-monad $\tDD$ is commutative.
\end{PropSub}
\begin{proof}
Recall that the $\sL$-monad $\tHH$ is induced by an $\sL$-adjunction \hbox{$\bK \nsststile{}{\rho} \bJ : \teL \hookrightarrow \aeL$} whose underlying ordinary adjunction $K \nsststile{}{\rho} J : \tL \hookrightarrow \sL$ carries the structure of a symmetric monoidal adjunction, with $\tL$ closed \pbref{thm:compl_sm_adj}.  This symmetric monoidal adjunction determines an $\sL$-adjunction \hbox{$\acute{K} \nsststile{}{\rho} \grave{J} : J_*\left(\aetL\right) \rightarrow \aeL$} whose induced $\sL$-monad $\tHH'$ is isomorphic to $\tHH$ \pbref{thm:compl_sm_adj}.   Hence we have that $\tDD = [\acute{F},\grave{G}](G_*(\tHH)) \cong [\acute{F},\grave{G}](G_*(\tHH')) =: \tDD'$, so it suffices (by \bref{rem:comm_inv_under_iso}) to show that $\tDD'$ is commutative.

Now $\tDD'$ is the $\X$-monad induced by the composite $\X$-adjunction
\begin{equation}
\xymatrix {
\aeX \ar@/_0.5pc/[rr]_{\acute{F}}^(0.4){\eta}^(0.6){\epsilon}^{\top} & & G_*\aeL \ar@/_0.5pc/[ll]_{\grave{G}} \ar@/_0.5pc/[rr]_{G_*\acute{K}}^(0.4){}^(0.6){}^{\top} & & {G_*J_*(\aetL)\;,} \ar@/_0.5pc/[ll]_{G_*(\grave{J})}
}
\end{equation}
which by \bref{thm:compn_assoc_enr_adj} is isomorphic to the $\X$-adjunction
\begin{equation}\label{eq:enradj_kf_gj}
\xymatrix {
\aeX \ar@/_0.5pc/[rr]_{\acute{(KF)}}^(0.4){}^(0.6){}^{\top} & & (GJ)_*\aetL \ar@/_0.5pc/[ll]_{\grave{(GJ)}}
}
\end{equation}
associated to the composite symmetric monoidal adjunction
$$
\xymatrix {
\X \ar@/_0.5pc/[rr]_F^(0.4){\eta}^(0.6){\epsilon}^{\top} & & \sL \ar@/_0.5pc/[ll]_G \ar@/_0.5pc/[rr]_K^(0.4){\rho}^(0.6){}^{\top} & & {\tL\;.} \ar@{_{(}->}@/_0.5pc/[ll]_J
}
$$
Hence $\tDD'$ is isomorphic to the $\X$-monad $\tDD''$ induced by \eqref{eq:enradj_kf_gj}. and by \bref{thm:enradj_assoc_smadj_ind_commmnd}, $\tDD''$ is commutative, so $\tDD'$ is commutative.
\end{proof}

\begin{LemSub} \label{thm:partial_fx_inv_by_h}
Suppose that the hypotheses of Theorem \bref{thm:general_fubini} hold.  Then for each object $X$ of $\X$, the morphism $\partial_{FX}:FX \rightarrow HFX = (FX)^{**}$ in $\sL$ is inverted by $H:\sL \rightarrow \sL$.
\end{LemSub}
\begin{proof}
Having chosen the cotensors $[X,R]$ in $G_*\aeL$ as in \bref{thm:assoc_enr_adj}.\textnormal{2}, we have that
$$[X,R] = \aeL(FX,R) = (FX)^*\;,$$
so since $[X,R]$ is reflexive, the unit component
$$\partial_{(FX)^*}:(FX)^* \rightarrow H((FX)^*) = (FX)^{***}$$
is an isomorphism.

But $\partial_{(FX)^*}$ has a retraction $\partial_{FX}^*$,
as follows.  Indeed, as
$$\partial_{(FX)^*}:(FX)^* \rightarrow \aeL(\aeL((FX)^*,R),R)$$
is the transpose of 
$$\Ev:(FX)^* \otimes \aeL((FX)^*,R) \rightarrow R\;,$$
we find that the composite
$$(FX)^* \xrightarrow{\partial_{(FX)^*}} \aeL(\aeL((FX)^*,R),R) \xrightarrow{\partial_{FX}^* \;=\; \aeL(\partial_{FX},R)} \aeL(FX,R) = (FX)^*$$
has transpose
$$(FX)^* \otimes FX \xrightarrow{1 \otimes \partial_{FX}} (FX)^* \otimes \aeL((FX)^*,R) \xrightarrow{\Ev} R\;,$$
which is equal to $\Ev:\aeL(FX,R) \otimes FX \rightarrow R$, so that $\partial_{FX}^* \cdot \partial_{(FX)^*} = 1_{(FX)^*}$ as needed.

Hence, since $\partial_{(FX)^*}$ is an isomorphism, its retraction $\partial_{FX}^*:(FX)^{***} \rightarrow (FX)^*$ is an isomorphism.  Applying $(-)^*:\sL^\op \rightarrow \sL$, we obtain an isomorphism
$$H\partial_{FX} = \partial_{FX}^{**} : HFX = (FX)^{**} \rightarrow (FX)^{****} = HHFX\;.$$
\end{proof}

\begin{PropSub} \label{thm:dist_via_compl}
Suppose that the hypotheses of Theorem \bref{thm:general_fubini} hold.  Then the morphism of $\X$-monads $GiF:\tDD \rightarrow \DD$ \eqref{eq:morph_td_to_d} is an isomorphism.
\end{PropSub}
\begin{proof}
It suffices to show that the underlying ordinary natural transformation $GiF : G\tH F = \tD \rightarrow D = GHF$ is an isomorphism.  Letting $X$ be an object of $\X$, it therefore suffices to show that $i_{FX} : \tH FX \rightarrow HFX$ is an isomorphism in $\sL$.  To this end, note that the periphery of the diagram
$$
\xymatrix@R=2ex {
{FX} \ar[r]^{\rho_{FX}} \ar[d]_{\partial_{FX}} & {\tH FX} \ar[d]^{i_{FX}} \\
{HFX} \ar@2{-}[r] \ar@{-->}[ur]^s                    & {HFX}
}
$$
commutes since $i \cdot \rho = \partial$ \pbref{thm:mnd_morph_i}.  By \bref{thm:partial_fx_inv_by_h}, $\partial_{FX} \in \Sigma_{H} $.  Also, the morphism $i_{FX}$ lies in $\teL = \aeL_{\Sigma_{H}}$, since both $\tH FX$ and $HFX$ lie in $\tL$ (the latter since $\bH $ factors through $\tL$ as $\bH = \bG^{\HH}\bF^{\HH} = \bJ \bP \bF^\HH$ \pbref{par:factn_of_em_adj}).  Hence, by \bref{thm:orth_of_morphs_in_orth_subcat}, we have that $\partial_{FX} \downarrow i_{FX}$, so there is a unique morphism $s$ in $\sL$ making the given diagram commute, whence in particular $i_{FX} \cdot s = 1_{HFX}$.  But $i_{FX}$ is a mono \pbref{thm:mnd_morph_i}, so $i_{FX}$ is an isomorphism.
\end{proof}

Hence, under the hypotheses of Theorem \bref{thm:general_fubini}, we have an isomorphism of $\X$-monads $\DD \cong \tDD$ with $\tDD$ commutative \pbref{thm:td_comm}, so $\DD$ is commutative (by \bref{rem:comm_inv_under_iso}) and Theorem \bref{thm:general_fubini} is proved.

\bibliographystyle{amsplain}
\bibliography{Fubini}

\providecommand{\bysame}{\leavevmode\hbox to3em{\hrulefill}\thinspace}
\providecommand{\MR}{\relax\ifhmode\unskip\space\fi MR }
\providecommand{\MRhref}[2]{%
  \href{http://www.ams.org/mathscinet-getitem?mr=#1}{#2}
}
\providecommand{\href}[2]{#2}
\begin{thebibliography}{10}

\bibitem{BeBu}
R.~Beattie and H.-P. Butzmann, \emph{Convergence structures and applications to
  functional analysis}, Kluwer Academic Publishers, Dordrecht, 2002.

\bibitem{Bou}
N.~Bourbaki, \emph{El\'ements de math\'ematique. {XIII}. {L}ivre {VI}:
  {I}nt\'egration.}, Actualit\'es Sci. Ind., no. 1175, Hermann et Cie, Paris,
  1952.

\bibitem{Bu}
H.-P. Butzmann, \emph{\"{U}ber die {$c$}-{R}eflexivit\"at von {$C_{c}(X)$}},
  Comment. Math. Helv. \textbf{47} (1972), 92--101.

\bibitem{CHK}
C.~Cassidy, M.~H{\'e}bert, and G.~M. Kelly, \emph{Reflective subcategories,
  localizations and factorization systems}, J. Austral. Math. Soc. Ser. A
  \textbf{38} (1985), no.~3, 287--329.

\bibitem{Cr}
G.~S.~H. Cruttwell, \emph{Normed spaces and the change of base for enriched
  categories}, Ph.D. thesis, Dalhousie University, 2008.

\bibitem{Day:Refl}
B.~Day, \emph{A reflection theorem for closed categories}, J. Pure Appl.
  Algebra \textbf{2} (1972), no.~1, 1--11.

\bibitem{Day:AdjFactn}
\bysame, \emph{On adjoint-functor factorisation}, Lecture Notes in Math.
  \textbf{420} (1974), 1--19.

\bibitem{Dub}
E.~J. Dubuc, \emph{Kan extensions in enriched category theory}, Lecture Notes
  in Mathematics, Vol. 145, Springer-Verlag, 1970.

\bibitem{Edw}
R.~E. Edwards, \emph{A theory of {R}adon measures on locally compact spaces},
  Acta Math. \textbf{89} (1953), 133--164.

\bibitem{EgMoSi}
J.~Egger, R.~E. M{\o}gelberg, and A.~Simpson, \emph{Enriching an effect
  calculus with linear types}, Lecture Notes in Comput. Sci. \textbf{5771}
  (2009), 240--254.

\bibitem{EiKe}
S.~Eilenberg and G.~M. Kelly, \emph{Closed categories}, Proc. {C}onf.
  {C}ategorical {A}lgebra ({L}a {J}olla, {C}alif., 1965), Springer, 1966,
  pp.~421--562.

\bibitem{FrKe}
P.~J. Freyd and G.~M. Kelly, \emph{Categories of continuous functors. {I}}, J.
  Pure Appl. Algebra \textbf{2} (1972), 169--191.

\bibitem{Gray}
J.~W. Gray, \emph{Formal category theory: adjointness for {$2$}-categories},
  Lecture Notes in Mathematics, Vol. 391, Springer-Verlag, 1974.

\bibitem{He}
H.~Herrlich, \emph{Regular categories and regular functors}, Canad. J. Math.
  \textbf{26} (1974), 709--720.

\bibitem{Ke:Doctr}
G.~M. Kelly, \emph{Doctrinal adjunction}, Lecture Notes in Math. \textbf{420}
  (1974), 257--280.

\bibitem{Ke:Ba}
\bysame, \emph{Basic concepts of enriched category theory}, Repr. Theory Appl.
  Categ. (2005), no.~10, Reprint of the 1982 original [Cambridge Univ. Press].

\bibitem{Kock:Comm}
A.~Kock, \emph{Monads on symmetric monoidal closed categories}, Arch. Math.
  (Basel) \textbf{21} (1970), 1--10.

\bibitem{Kock:SmComm}
\bysame, \emph{Strong functors and monoidal monads}, Arch. Math. (Basel)
  \textbf{23} (1972), 113--120.

\bibitem{Kock:ProResSynthFuncAn}
\bysame, \emph{Some problems and results in synthetic functional analysis},
  Category theoretic methods in geometry ({A}arhus, 1983), Various Publ. Ser.
  (Aarhus), vol.~35, Aarhus Univ., 1983, Available at
  \verb|http://home.imf.au.dk/kock/PRSFA.pdf|, pp.~168--191.

\bibitem{Kock:Dist}
\bysame, \emph{Commutative monads as a theory of distributions}, Theory Appl.
  Categ. \textbf{26} (2012), No. 4, 97--131.

\bibitem{Law:IntroCatsContPhys}
F.~W. Lawvere, \emph{Introduction}, Categories in continuum physics ({B}uffalo,
  {N}.{Y}., 1982), Lecture Notes in Math., vol. 1174, Springer, 1986,
  pp.~1--16.

\bibitem{Law:VoFu}
\bysame, \emph{Volterra's functionals and covariant cohesion of space}, Rend.
  Circ. Mat. Palermo (2) Suppl. (2000), no.~64, 201--214, Revised version of
  1998 available at \verb|http://www.acsu.buffalo.edu/~wlawvere/Volterra.pdf|.

\bibitem{Law:AxEd}
\bysame, \emph{Foundations and applications: axiomatization and education},
  Bull. Symbolic Logic \textbf{9} (2003), no.~2, 213--224.

\bibitem{MacL}
S.~Mac~Lane, \emph{Categories for the working mathematician}, second ed.,
  Springer-Verlag, 1998.

\bibitem{McCNta}
R.~A. McCoy and I.~Ntantu, \emph{Topological properties of spaces of continuous
  functions}, Lecture Notes in Mathematics, vol. 1315, Springer-Verlag, 1988.

\bibitem{Nel:EnrAlg}
L.~D. Nel, \emph{Enriched algebraic categories with applications in functional
  analysis}, Lecture Notes in Math. \textbf{915} (1982), 247--259.

\bibitem{Rie}
E.~Riehl, \emph{Categorical homotopy theory}, Lecture notes,
  \verb|http://www.math.harvard.edu/~eriehl/266x/lectures.pdf|.

\bibitem{Seal:Tensor}
G.~J. Seal, \emph{Cartesian closed topological categories and tensor products},
  Appl. Categ. Structures \textbf{13} (2005), no.~1, 37--47.

\bibitem{Wy:QuTo}
O.~Wyler, \emph{Lecture notes on topoi and quasitopoi}, World Scientific
  Publishing Co. Inc., 1991.

\end{thebibliography}

\end{document}